\documentclass[12pt]{amsart}

\usepackage{amssymb}

\usepackage{enumitem}

\usepackage{graphicx}
\usepackage{xcolor}
\usepackage{comment}
\usepackage{cite}
\makeatletter
\@namedef{subjclassname@2010}{%
  \textup{2020} Mathematics Subject Classification}
\makeatother

\usepackage[T1]{fontenc}

\usepackage{amsmath,amssymb,amsthm,amsfonts,latexsym,amscd,hyperref}
\allowdisplaybreaks
\hypersetup{colorlinks,linkcolor={blue},citecolor={blue},urlcolor={blue}}
\usepackage{xcolor}
\usepackage{mathtools}
\newtheorem{thm}{Theorem}[section]
\newtheorem{cor}[thm]{Corollary}
\newtheorem{lem}[thm]{Lemma}
\newtheorem{prop}[thm]{Proposition}
\theoremstyle{definition}
\newtheorem{defn}[thm]{Definition}
\theoremstyle{remark}
\newtheorem{rem}[thm]{Remark}
\numberwithin{equation}{section}
\newcommand{\vertiii}[1]{{\left\vert\kern-0.25ex\left\vert\kern-0.25ex\left\vert #1 
    \right\vert\kern-0.25ex\right\vert\kern-0.25ex\right\vert}}

\begin{document}


\baselineskip=17pt


\title[]{Dunkl paraproducts and fractional Leibniz rules for the Dunkl Laplacian}

\author[THE ANH BUI]{The Anh Bui}
\address{School of Mathematical and Physical Sciences, Macquarie University, NSW 2109,
	Australia}
\email{the.bui@mq.edu.au}

\author[SUMAN MUKHERJEE]{Suman Mukherjee$^\dagger$}
\address{Department of Mathematics, Indian Institute of Technology Bombay, Powai, Mumbai--400076, India.}
\email{sumanmukherjee822@gmail.com}

\date{}

\begin{abstract}
We establish fractional Leibniz rules for the Dunkl Laplacian $\Delta_k$ of the form  
$$\|(-\Delta_k)^s(fg)\|_{L^p(d\mu_k)} \lesssim \|(-\Delta_k)^s f\|_{L^{p_1}(d\mu_k)} \|g\|_{L^{p_2}(d\mu_k)} + \|f\|_{L^{p_1}(d\mu_k)} \|(-\Delta_k)^s g\|_{L^{p_2}(d\mu_k)}.$$  
Our approach relies on adapting the classical paraproduct decomposition to the Dunkl setting. In the process, we develop several new auxiliary results. Specifically, we show that for a Schwartz function $f$, the function $(-\Delta_k)^s f$ satisfies a pointwise decay estimate; we establish a version of almost orthogonality estimates adapted to the Dunkl framework; and we investigate the boundedness of Dunkl paraproduct operators on the Lebesgue spaces.
\end{abstract}

\thanks{$\dagger$ \tt{Corresponding author}}
\maketitle
\pagestyle{myheadings}
\markboth{THE ANH BUI AND SUMAN MUKHERJEE}{DUNKL PARAPRODUCTS AND FRACTIONAL LEIBNIZ RULES}


\section{Introduction and statement of the main results}\label{sec intro}

Let $\Delta$ denote the Laplacian operator on $\mathbb{R}^d$. Then for any function $f$ in the Schwartz space $\mathcal{S}(\mathbb{R}^d)$, and for any $s>0$, the fractional powers of the Laplacian is defined by the relation 
\begin{equation}\label{defn of usual frac powers of laplacian}
(-\Delta)^{s}f(x)=\int_{\mathbb{R}^d}|\xi|^{2s}\,\mathcal{F}f(\xi)\,e^{\langle i\xi,\,x\rangle}\,d\xi,
\end{equation}
 where $\mathcal{F}f$ denote the Fourier transform of the function $f$ on $\mathbb{R}^d$.
An estimate of the form
\begin{equation}\label{classical Koto_ponce ineq}
\|(-\Delta)^s(fg)\|_{L^{p}} \lesssim \|(-\Delta)^s f\|_{L^{p_1}}\, \|g\|_{L^{p_2}} + \|f\|_{L^{\widetilde{p}_1}}\, \|(-\Delta)^s g\|_{L^{\widetilde{p}_2}}
\end{equation}
is known as the \emph{fractional Leibniz rule} or \emph{the Kato-Ponce inequality}. Here, $f$ and $g$ $\in \mathcal{S}(\mathbb{R}^d)$, and the indices satisfy
$1 < p,\, p_1,\, p_2,\, \widetilde{p}_1,\, \widetilde{p}_2 < \infty$, with  $1/p=1/p_1+1/p_2=1/\widetilde p_1 + 1/\widetilde p_2$, and either $2s > \max\{0,\, d/p - d\}$ or $s \in \mathbb{Z}^+$. The first version of this inequality appeared in the work of Kato–Ponce \cite{KatoCEATEANSE}. After that, Gulisashvili–Kon \cite{GulisashviliESPOSS} proved it for the range $p > 1$. The range $1/2 < p \leq 1$ was later treated by Grafakos-Oh \cite{GrafakosTKPI}. The proof of this inequality mainly relies on the Coifman-Meyer bilinear multiplier theorem \cite{CoifmenNHAOTAPIBL}. A different proof was later given by Muscalu-Pipher-Tao-Thiele \cite{MuscaluBPP} (see also \cite{MuscaluBook2CAMHA}), using paraproduct decomposition.
\par The inequality \eqref{classical Koto_ponce ineq}, along with its various forms, has become a fundamental tool in the analysis of several nonlinear partial differential equations. These include the Euler and Navier-Stokes equations \cite{KatoCEATEANSE}, the KdV equation \cite{ChristDOSASOTGKDVE, KenigWPASRFTGKDVEVTCP}, and studies on the smoothing effects of Schr\"odinger semigroups \cite{GulisashviliESPOSS}. Due to its wide applicability, the Kato-Ponce inequality has attracted considerable attention and has been investigated by numerous authors across a broad range of different settings. We present below a selected overview of significant developments in the relevant literature:
\begin{itemize}
    \item The endpoint Kato-Ponce inequality was studied by Grafakos-Maldonado-Naibo \cite{GrafakosARONEKPI} and Bourgain-Li \cite{BourgainOAEKPI}.
    \item Extensions to weighted and variable Lebesgue spaces were considered by Cruz-Uribe-Naibo \cite{CruzuribeKPIOWAVLS}.
    \item The inequality in Besov spaces was explored by Brummer-Naibo \cite{BrummerBOWHSSMAKPI}.
    \item A higher-order version of the fractional Leibniz rule was established by Fujiwara-Georgiev-Vladimir \cite{FujiwaraHOFLR}.
    \item An improved variant of the Kato-Ponce inequality was introduced by Li \cite{LiOKPAFL}.
    \item The $L^1$ endpoint case was addressed by Oh-Wu \cite{OhOLEPKPI}.
    \item Weighted endpoint versions were studied by Wu \cite{WuWEFLR}.
    \item Kato-Ponce inequalities with polynomial weights were investigated by Oh-Wu \cite{OhTKPIWPW}.
    \item Fractional Leibniz rules associated with bilinear Hermite pseudo-multipliers were examined by Ly-Naibo \cite{LyFLRATBHPM}.
    \item Generalizations to quasi-Banach function spaces were considered by Hale-Naibo \cite{HaleFLRITSOQBFS}.
    \item Fractional Leibniz-type rules on spaces of homogeneous type were analyzed by Liu-Zhang \cite{LiuFLROSOHT}.
    \item Weighted Kato-Ponce inequalities for multiple factors were developed by Douglas-Grafakos \cite{DouglasWKPIFMF}.
    \item Fractional Leibniz rule on the torus was studied by B\'{e}nyi-Oh-Zhao \cite{BenyiFLROTT}.
    \item Weighted biparameter versions were introduced by Hale-Naibo \cite{HaleWBFLR}.
    \item In the setting of Carnot groups, the theory was explored by Fang-Li-Zhao \cite{FangTFLROTPSOCG}.
\end{itemize}

\par We now shift our focus in a different direction. Over the past 35 years, a parallel framework to classical Fourier analysis has emerged in Euclidean spaces, known as \emph{analytic Dunkl theory}. In this setting, Fourier analysis is developed with respect to a root system $R$, a corresponding reflection group $G$, and a multiplicity function $k$ (see Section~\ref{sec prel} for notation and definitions). The Dunkl transform of a function $f \in \mathcal{S}(\mathbb{R}^d)$ is defined analogously to the classical Fourier transform, and is given by
$$\mathcal{F}_k f(\xi) = \int_{\mathbb{R}^d} f(x) E_k(-i\xi, x)\, d\mu_k(x),$$
where $E_k$ denotes the Dunkl kernel and $\mu_k$ is the associated measure. When the multiplicity function $k$ is identically zero, $\mathcal{F}_k$ reduces to the classical Fourier transform $\mathcal{F}$. In fact, when $k \equiv 0 $, the Dunkl setting reduces to the classical case. Thus, the classical setting can be viewed as a special case of the Dunkl framework.
\par Let $\Delta_k$ denote the Dunkl Laplacian on $\mathbb{R}^d$. Then similar to the classical case, in view of the identity $\mathcal{F}_k(\Delta_k f)(\xi) = -|\xi|^2 \mathcal{F}_k f(\xi)$, the fractional powers of the Dunkl Laplacian can be defined for any $s > 0$ and $f \in \mathcal{S}(\mathbb{R}^d)$ as
$$(-\Delta_k)^{s}f(x)=\int_{\mathbb{R}^d}|\xi|^{2s}\,\mathcal{F}_kf(\xi)E_k(i\xi,x)\,d\mu_k(\xi).$$

\par Naturally, a fractional Leibniz rule in this context has also been studied. In fact, when the associated reflection group $G$ is isomorphic to $\mathbb{Z}_2^d$, the following result was established by Wr\'{o}bel.
\begin{thm}\cite[Corollary 4.3]{WrobelABMVAFC}\label{wrobels thm}
Let $1 < p, p_1, p_2 < \infty$ be such that $1/p = 1/p_1 + 1/p_2$. Then for any $s > 0$ and $f,\, g \in \mathcal{S}(\mathbb{R}^d)$, with at least one of $f$ or $g$ being $\mathbb{Z}_2^d$-invariant, we have
 \begin{eqnarray*}
||(-\Delta_k)^s(fg)||_{L^{p}(d\mu_k)} &\lesssim & ||(-\Delta_k)^s(f)||_{L^{p_1}(d\mu_k)}\, ||g||_{L^{p_2}(d\mu_k)}
+\  ||f||_{L^{ p_1}(d\mu_k)}\, ||(-\Delta_k)^s(g)||_{L^{ p_2}(d\mu_k)}.
\end{eqnarray*}
\end{thm}
\emph{However, we would like to point out that there appears to be an error in \cite[Lemma 4.4]{WrobelABMVAFC}. Consequently, it is not clear whether Theorem \ref{wrobels thm} holds as stated. This issue has already been discussed with the author of \cite{WrobelABMVAFC} by the second named author of the present paper, and was confirmed via personal communication \cite{Wrobel2}.}
\par In this article, we aim to address this issue by establishing a fractional Leibniz rule for the Dunkl Laplacian. In fact, we prove more general results by considering an arbitrary reflection group $G$ and without assuming that either of the functions is $G$-invariant. We now present our main results concerning fractional Leibniz-type rules.

\begin{thm}\label{frac leib rule main thm}
Let $1<p, p_1,p_2, \widetilde p_1,\widetilde p_2<\infty$  satisfy the relation $1/p=1/p_1+1/p_2=1/\widetilde p_1 + 1/\widetilde p_2$. Then for any $s>0$ and $f,\,g\in \mathcal{S}(\mathbb{R}^d )$, we have
 \begin{eqnarray*}
||(-\Delta_k)^s(fg)||_{L^{p}(d\mu_k)} &\lesssim & ||(-\Delta_k)^s(f)||_{L^{p_1}(d\mu_k)}\, ||g||_{L^{p_2}(d\mu_k)}\\
&&+\  ||f||_{L^{\widetilde p_1}(d\mu_k)}\, ||(-\Delta_k)^s(g)||_{L^{\widetilde p_2}(d\mu_k)}.
\end{eqnarray*}
\end{thm}

\begin{rem}
Theorem \ref{frac leib rule main thm} recovers the fractional Leibniz rule for the classical Laplacian as a special case (see, e.g., \cite[Theorem 7.6.1]{GrafakosModernBook} and \cite[Theorem 1]{GrafakosTKPI}).
\end{rem}
\par In the classical case, using the definition \eqref{defn of usual frac powers of laplacian} and the translation invariance of the Lebesgue measure, one can easily verify that applying $(-\Delta)^s$ to the product of two functions yields a bilinear Fourier multiplier operator with symbol $|\xi + \eta|^{2s}$. The main difficulty in extending this approach to the Dunkl setting lies in the lack of translation invariance of the associated measure, which prevents a direct application of Fourier multiplier techniques. To overcome this, Wr\'{o}bel \cite{WrobelABMVAFC} proposed an alternative method by restricting to the case where one of the functions in the product is invariant under the action of the reflection group. However, as discussed earlier, this approach is also unsuccessful due to technical issues. This observation suggests the necessity of abandoning such methods in favor of an approach where $(-\Delta_k)^s$ acts on a single function, rather than on the product of two functions.
\par One such approach is available in the classical setting due to the work of Muscalu-Pipher-Tao-Thiele \cite{MuscaluBPP}, where they employ a `paraproduct decomposition'. This method, which involves convolutions of the product of two functions with suitable kernels, allows the fractional Laplacian $(-\Delta)^s$ to act only on one of the functions. So, this method appears to be more suitable in the Dunkl setting. In this setup, however, no analogous analysis has been carried out. To adapt their approach in this context, we introduce the following definition.
\begin{defn}\label{defn of paraproducts}
  Let  $\psi, \phi$ and $\theta \in C_c^{\infty}(\mathbb{R}^d )$ be such that $\psi=\mathcal{F}_k\Psi, \phi=\mathcal{F}_k\Phi,\theta=\mathcal{F}_k \Theta$. For any $j\in\mathbb{Z},$ define $\Psi_j(\xi)=2^{jd_k}\Psi(2^j\xi), \Phi_j(\xi)=2^{jd_k}\Phi(2^j\xi)$ and $\Theta_j(\xi)=2^{jd_k}\Theta(2^j\xi)$. Then for any $f, g\in \mathcal{S}(\mathbb{R}^d )$, the \emph{Dunkl-paraproduct operator} $\Pi[\theta, \psi, \phi]$ is given by
  \begin{equation*}
 \Pi[\theta, \psi, \phi](f, g)(x)=\sum\limits_{j\in \mathbb{Z}} \Theta_j *_k\big((\Psi_j*_kf)(\cdot)(\Phi_j*_kg)(\cdot)\big)(x),  
  \end{equation*}
  where $d_k$ is the homogeneous dimension of the measure $\mu_k$, and $*_k$ is the Dunkl convolution operator (see Section \ref{sec prel}). 
\end{defn}
\par Our approach is to establish the boundedness of these paraproduct operators as bilinear operators from a product of Lebesgue spaces to a Lebesgue space, and subsequently use this to prove fractional Leibniz rules. Specifically, we first prove the following result.
 \begin{thm}\label{MAIN bddnees of 3paraproducts}
For any $\psi, \phi, \theta \in C_c^{\infty}(\mathbb{R}^d)$, if the supports of at least two among $\psi$, $\phi$, and $\theta$ do not contain $0$, then the Dunkl paraproduct operator $\Pi[\theta, \psi, \phi]$ is bounded from $L^{p_1}(\mathbb{R}^d,$ $d\mu_k)\times L^{p_2}(\mathbb{R}^d,\,d\mu_k)$ to $L^{p}(\mathbb{R}^d,\, d\mu_k)$ for all $p,p_1,p_2$ satisfying $1<p_1,p_2<\infty$ with $1/p=1/p_1+1/p_2$.  \end{thm}
  Definition~\ref{defn of paraproducts} involves the Dunkl convolution, which is not well studied, and relatively little is known about its properties. Consequently, proving such boundedness results requires substantial effort. We achieve this by employing the recently developed theory of bilinear Dunkl-Calderón-Zygmund operators \cite{SumanWBMTIDSVSI}, along with some estimates for the Dunkl translations of Schwartz functions (Proposition \ref{smooth function all decay}) and almost orthogonality type estimates (Proposition \ref{almost orhto esti thm}), which we establish in this work.
\begin{rem}
The proofs of Theorem \ref{frac leib rule main thm} relies on the boundedness of paraproduct operators, as established in Theorem \ref{MAIN bddnees of 3paraproducts}. However, there is a key difference: while the boundedness of paraproduct operators holds for the full range of~$p$, the fractional Leibniz rules require the restriction $p > 1$. This limitation arises because the proof involves an operator $\widetilde\Pi[\theta, \psi, \phi]$, which differs slightly from the standard paraproducts as in Definition \ref{defn of paraproducts}. In particular, one of the functions among $\theta$, $\phi$, and $\psi$ is not smooth here. For such an operator, our method does not suffice to establish boundedness $L^{p_1}(\mathbb{R}^d,\, d\mu_k) \times L^{p_2}(\mathbb{R}^d,\, d\mu_k) \to L^{p}(\mathbb{R}^d,\, d\mu_k)$ when $p \leq 1$. It will be challenging and interesting to work out the case $p \leq 1$.
\end{rem}
 Once the boundedness of the paraproduct operators is established, two key steps remain. The first is to express the product of two functions in terms of paraproducts. This requires detailed information about the support properties of Dunkl translations, which we establish in Proposition \ref{decompo of fg in paraproduct}. The second step involves showing that, for a Schwartz function $f$, $(-\Delta_k)^s f$ exhibits a certain pointwise decay, ensuring that $(-\Delta_k)^s f \in L^p(\mathbb{R}^d,\, d\mu_k)$. This is proved using semigroup techniques in Proposition \ref{frac deri in Lp thm}.
\par We conclude this introduction with a brief outline of the paper and some remarks on notation. In Section \ref{sec prel}, we introduce the notation and preliminary results from Dunkl theory that will be used throughout the paper. Section \ref{sec main ingredients} is devoted to developing the key ingredients needed for proving the boundedness of the paraproduct operators. Specifically, we establish estimates for the Dunkl translations of Schwartz functions, prove almost orthogonality estimates, and present the relevant results from the theory of bilinear Dunkl-Calder\'on-Zygmund operators. In Section \ref{sec Dunkl paraproduct}, we prove Theorem \ref{MAIN bddnees of 3paraproducts}, which establishes the boundedness of the paraproducts. Finally, in Section \ref{sec frac Leib rule}, we first establish the paraproduct decomposition and then prove the fractional Leibniz rules, that is, Theorems \ref{frac leib rule main thm}.
\par The notations introduced in Definition \ref{defn of paraproducts} will be used throughout the paper. In particular, for a suitable function $\psi$, we denote by $\Psi$ the function such that $\mathcal{F}_k \Psi = \psi$. We also use the scaled versions $\psi_j(\xi) = \psi(\xi / 2^j)$ and $\Psi_j(\xi) = 2^{j d_k} \Psi(2^j \xi)$. Similar notation is used for the functions $\phi$ and $\theta$.

\section{Preliminaries of Dunkl theory}\label{sec prel}
Let $x, y\in \mathbb{R}^d$, $\langle x, y \rangle$ be the Euclidean inner product and $|x|:= \sqrt{\langle x,x \rangle} $. For $0\neq \lambda \in \mathbb{R}^d$, define the function $\sigma_\lambda$ given by
$$\sigma_{\lambda}(x)=x-2\frac{\langle x,\lambda \rangle}{|\lambda|^2}\lambda.$$
Then $\sigma_\lambda$ is known as the \emph{reflection}
with respect to the hyperplane $\lambda^{\bot}:=\{x\in \mathbb{R}^d : \langle x, \lambda \rangle=0 \}$. Consider a finite subset $R$ of $\mathbb{R}^d$ such that $0\notin R$. Then $R$ is called a \emph{root system} if for any $\lambda \in R$, we have $R \cap \mathbb{R}\lambda = \{ \pm \lambda \} $ and $\sigma _{\lambda} (R) = R $. Here we fix a normalized root system $R$ such that  $|\lambda|=\sqrt{2},\ \forall \lambda \in R$. The set of reflections $\{ \sigma_ \lambda : \lambda \in R\}$ generate a finite subgroup $G$ of $O(d, \mathbb{R})$, which is known as the \emph{reflection group} (or  \emph{Coxeter group}) associated with the root system $R$. A function $k :R\rightarrow \mathbb{C}$ which is $G$-invariant, is known as a  \emph{multiplicity function}. Throughout the paper, we consider a fixed multiplicity function $k\geq 0$. Let $\gamma_k=\sum_{\lambda \in R}k(\lambda)$, $d_k=d+\gamma_k $, $h_k$ be the $G$-invariant homogeneous function given by $h_k(x) = \prod_ {\lambda \in R} | \langle x , \lambda \rangle|^{k (\lambda)}$ and $d\mu_k(x)$ be the normalized measure $c_k h_k(x)dx$, where 
$$c_k^{-1}=\int_{\mathbb{R}^d}e^{-{|x|^2}/{2}}\,h_k(x)\,dx.$$
\par Let $B(x,r)$ be a ball in $\mathbb{R}^d$, centred at $x$ and of radius $r$. Then, with respect to this measure, the volume $\mu_k(B(x,r))$ of $B(x,r)$ is given by 
\begin{equation}\label{Volofball}
    \mu_k(B(x,r))\sim \ r^d\prod\limits_{\lambda \in R}\left(|\langle x,\lambda \rangle|+r\right)^{k(\lambda)}.
\end{equation}
If $r_2>r_1>0$, then it is easy to see from above that
\begin{eqnarray}\label{VOLRADREL}
   C\left(\frac{r_1}{r_2}\right)^{d_k}\leq  \frac{\mu_k(B(x,r_1))}{\mu_k(B(x,r_2))}\leq C^{-1} \left(\frac{r_1}{r_2}\right)^{d}.
\end{eqnarray}
\subsection{Orbit distances and orbit of balls}
Let us define $d_G(x,y)$ to be the distance between the $G$-orbits of $x$ and $y$, that is $d_G(x,y)=\min\limits_{\sigma \in G}|\sigma(x)-y|$, and for any $r>0$ we write $V_G(x,y,r)=\max\,\left\{\mu_k(B(x,r)),\mu_k(B(y,r))\right\}$. Then an immediate consequence of (\ref{Volofball}) is that
\begin{equation}\label{V(x,y,d(x,y) comparison}
V_G(x,y,d_G(x,y))\sim \mu_k(B(x,d_G(x,y))\sim \mu_k(B(y,d_G(x,y)).
\end{equation}
Let $\mathcal{O}(B)$ denotes the orbit of the ball $B$, that is 
$$\mathcal{O}(B)=\Big\{y\in \mathbb{R}^d: d_G(c_B,y)\leq r(B)\Big\}=\bigcup\limits_{\sigma \in G}\sigma (B),$$ 
where $c_B$ denotes the centre and $r(B)$ denotes the radius of the ball $B$. Then we have the relation
\begin{equation}\label{orbit volume compa}
    \mu_k(B)\leq \mu_k(\mathcal{O}(B))\leq |G|\,\mu_k(B).
\end{equation}
 $d_G$ is known as \emph{`the Dunkl metric'}, though it is \emph{not} a metric on $\mathbb{R}^d$. It satisfies the triangle inequality and also for any $x,y\in \mathbb{R}^d$, we have $d_G(x,y)\leq |x-y|.$

\subsection{Dunkl operators} 
The \emph{differential-difference operators} or the \emph{Dunkl operators} $T_{\xi}$ introduced by  C.F. Dunkl \cite{Dunkl} is given by
\begin{eqnarray*}
T_{\xi} f(x)= \partial_{\xi} f(x)+\sum\limits_{\lambda \in R} \frac{k(\lambda )}{2} \langle\lambda, \xi \rangle \frac{f(x) - f(\sigma _\lambda  x)}{\langle \lambda  , x \rangle}.
\end{eqnarray*}The Dunkl operators $T_{\xi}$ are the $k$-deformations of the directional derivative operator $\partial_{\xi}$ and coincides with them in the case $k=0$. In contrast to the classical directional derivatives, Dunkl operators do not satisfy usual Leibniz rule. However, if at least one of $f$ or $g$ is $G$-invariant then the Leibniz-type rule holds
 $$T_\xi(fg)(x)=T_\xi f(x)\, g(x)+f(x)T_\xi g(x).$$
\par Let $\{e_j : j = 1, 2, \cdots, d\}$ be the standard basis of $\mathbb{R}^d$ and set $T_j = T_{e_j}$ and $\partial_{j}=\partial_{e_j}$. We use $\Delta_k$ to denote the \emph{Dunkl Laplacian} on $\mathbb{R}^d$, that is 
$$ \Delta_k=\sum\limits^d_{j=1}T^2_{j}=\Delta f(x)
+ \sum_{\lambda\in R} k(\lambda)
\left(
\frac{\langle\nabla f(x),\lambda\rangle}{\langle\lambda,x\rangle}
-\frac{|\lambda |^2}{2}
\frac{f(x)-f(\sigma_\lambda x)}{\langle\lambda,x\rangle^2}
\right).$$
\subsection{Dunkl kernel and Dunkl transform}
 For a fixed $y\in \mathbb{R}^d$, it is known that there is a unique real analytic solution for the system
$T_{\xi} f =  \langle y,\xi \rangle  f $  satisfying $f(0)=1$. The solution $f(x)=E_k(x,y)$ is called the \emph{Dunkl kernel}. The Dunkl kernel $E_k(x,y)$ which is generalization of  the exponential function $e^{<x,y>}$, has a unique
extension to a holomorphic function on $\mathbb{C}^d\times \mathbb{C}^d$. We list below few properties of the Dunkl kernel (see \cite{RoslerDOTA, DunklIKWRGI} for details).\\
$\bullet$ $E_k(x,y)=E_k(y,x)$ for any $x,\ y\in \mathbb{C}^d,$\\
$\bullet$ $E_k(tx,y)=E_k(x,ty)$ for any $x,\ y\in \mathbb{C}^d$ and $t\in \mathbb{C},$\\
$\bullet$ $\overline{E_k(x,y)}=E_k(\overline{x},\overline{y})$ for any $x,\ y\in \mathbb{C}^d$,\\
$\bullet$ $|E_k(ix,y)|\leq 1$ for any $x,\ y\in \mathbb{R}^d$.
\par For $0<p< \infty$, let $L^p(\mathbb{R}^d,d\mu_k)$ denote the space of complex valued measurable functions $f$ such that 
$$||f||_{L^p(d\mu_k)}:=\Big(\int_{\mathbb{R}^d}|f(x)|^p\,d\mu_k(x)\Big)^{1/p}<\infty$$
and  $L^{p,\,\infty}(\mathbb{R}^d,d\mu_k)$ be the corresponding weak space with norm
$$||f||_{L^{p,\,\infty}(d\mu_k)}:=\sup\limits_{t>0}\,t\,[\,\mu_k\big(\{x\in \mathbb{R}^d:|f(x)|>t\}\big)]^{1/p}<\infty.$$
Also, let us use the notation $L^\infty(\mathbb{R}^d,d\mu_k)$ to denote the usual $L^\infty(\mathbb{R}^d)$ space with norm $\|\cdot\|_{L^\infty}$.
 For any $f$ $\in L^1(\mathbb{R}^d,d\mu_k)$ the \emph{Dunkl transform} of $f$ is defined by
\begin{eqnarray*}
\mathcal{F}_kf(\xi)=\int_{\mathbb{R}^d}f(x)E_k(-i\xi,x)\,d\mu_k(x).
\end{eqnarray*}
The following  are some well known properties of Dunkl transform \cite{DunklHTATFRG, deJeuTDT}.\\
$\bullet$ $\mathcal{F}_k$ is a homeomorphism of the space $\mathcal{S}(\mathbb{R}^d )$,\\
$\bullet$ If the function $f$ is radial then so is  $\mathcal{F}_kf$,\\
$\bullet$ $\mathcal{F}_k$ extends to an isometry on  $L^2(\mathbb{R}^d,d\mu_k)$ (Plancherel's formula), that is,
$$||\mathcal{F}_kf||_{L^2(d\mu_k)}=||f||_{L^2(d\mu_k)},$$
$\bullet$ If both $f$ and $\mathcal{F}_kf$ are in $L^1(\mathbb{R}^d,d\mu_k)$, then the following Dunkl inversion formula holds
$$f(x)=\mathcal{F}_k^{-1}(\mathcal{F}_kf)(x):=\int_{\mathbb{R}^d}\mathcal{F}_kf(\xi)E_k(i\xi,x)\,d\mu_k(\xi),$$
$\bullet$ From definition of the Dunkl kernel, for any $f\in \mathcal{S}(\mathbb{R}^d )$, the following relations holds :
$$T_j\mathcal{F}_kf(\xi)=-\mathcal{F}_k(i(\cdot)_jf)(\xi) \text{ and } \mathcal{F}_k(T_jf)(\xi)=i\xi_j \mathcal{F}_kf(\xi).$$
\subsection{Dunkl translations Dunkl convolution}\label{sec dunkl trans and conv}
 The \emph{Dunkl translation} $\tau^k_xf$ of a function $f\in  L^2(\mathbb{R}^d,d\mu_k)$ is defined  in terms Dunkl transform by 
$$\mathcal{F}_k(\tau^k_xf)(y)=E_k(ix,y)\mathcal{F}_kf(y).$$
Since $E_k(ix,y)$ is bounded, the above formula defines $\tau^k_x$ as a bounded operator on $L^2(\mathbb{R}^d,d\mu_k).$
We list below few properties of the Dunkl translations (see \cite{ThangaveluCOMF} for details).\\
$\bullet$ For $f\in \mathcal{S}(\mathbb{R}^d )$, $\tau^k_x$ can be pointwise defined as 
$$\tau^k_xf(y)=\int_{\mathbb{R}^d}E_k(ix,\xi)E_k(iy,\xi)\mathcal{F}_kf(\xi)\,d\mu_k(\xi),$$
$\bullet$ $\tau_y^kf(x)=\tau_{x}^kf(y)$ for any $f$ in $\mathcal{S}(\mathbb{R}^d )$,\\
$\bullet$ $\tau^k_x(f_t)=(\tau^k_{t^{-1}x}f)_t,\ \forall x\in\mathbb{R}^d$ and $\forall f\in\mathcal{S}(\mathbb{R}^d)$, where $f_t(x)=t^{-d_k}f(t^{-1}x)$ and $\ t>0$.\\
$\bullet$ If $f$ is a reasonable radial function such that $f \geq 0$, then $\tau^k_x f \geq 0$.\\
$\bullet$ Although $\tau^k_x$ is bounded operator for radial functions in $L^p(\mathbb{R}^d,d\mu_k)$ (see \cite{GorbachevPLBDTGTOAIA}), it is not known whether Dunkl translation is bounded operator or not on whole $L^p(\mathbb{R}^d,d\mu_k)$ for $p\neq2.$

\par For $f,g \in L^2(\mathbb{R}^d, d\mu_k)$, the \emph{Dunkl convolution} $f*_kg$
of $f$ and $g$ is defined by 
$$f*_kg(x)=\int_{\mathbb{R}^d}f(y)\tau^k_xg(-y)\,d\mu_k(y).$$
$*_k$ has the following analogous properties (see \cite{ThangaveluCOMF} for details) to that of usual convolution.\\
$\bullet$ $f*_kg(x)=g*_kf(x)$ for any $f, g \in  L^2(\mathbb{R}^d, d\mu_k)$;\\
$\bullet$ $\mathcal{F}_k(f*_kg)(\xi)=\mathcal{F}_kf(\xi) \mathcal{F}_kg(\xi)$ for any $f, g \in  L^2(\mathbb{R}^d, d\mu_k)$.
\subsection{Dunkl heat kernel}
The family $\{e^{t \Delta_k}\}_{t \geq 0}$, referred to as the \emph{Dunkl heat semigroup}, is defined through the relation  
$$e^{t \Delta_k} f(x) = \mathcal{F}_k^{-1} \left( e^{-t|\cdot|^2} \mathcal{F}_k f \right)(x).$$  
This expression represents the solution to the \emph{Dunkl heat equation}, which is formulated as  
\begin{align*}
    \begin{cases}
       \partial_t u(x,t) - \Delta_k u(x,t) = 0, & \quad (x,t) \in \mathbb{R}^d \times (0,\infty),\\
       u(x,0) = f(x). &
    \end{cases}
\end{align*}  
By transforming the equation through the Dunkl transform, one can express the semigroup action as 
$$e^{t \Delta_k} f(x) = f \ast_k h_t(x) = \int_{\mathbb{R}^d} h_t(x,y)\, f(y)\, d\mu_k(y),$$  
where the function $h_t(x,y):=\tau^k_x h_t(-y)$, called the \emph{Dunkl heat kernel}, has the explicit form  
$$h_t(x) = \mathcal{F}_k^{-1}(e^{-t|\cdot|^2})(x) = (2t)^{-d_k/2} e^{-{|x|^2}/{(4t)}}.$$  
Clearly, $\tau^k_xh_t$ is a non-negative function belonging to the Schwartz class. We also record the following two properties of the Dunkl heat kernel for later use
\begin{equation}\label{integration heta kerenl is 1}
 \int_{\mathbb{R}^d}  h_t(x, y)\, d\mu_k(y)=1, \text{ for any } x\in \mathbb{R}^d
\end{equation}
 and 
\begin{eqnarray}\label{decay esti of heta kerenl with mod}
    h_t(x, y) &\lesssim & \left(1+\frac{|x-y|}{\sqrt{t}}\right)^{-2} \frac{1}{\mu_k(B(x, \sqrt{t}))}\, \exp\Big(-\frac{d_G(x,y)^2}{ct}\Big) \\
    &\lesssim & t^{-d_k/2}\, \exp\Big(-\frac{d_G(x,y)^2}{ct}\Big).\nonumber
\end{eqnarray}
The preceding results have been stated directly from \cite[Theorem 3.1]{HejnaROADFTHSHITRDS} and \cite[Lemma 4.13]{RoslerDOTA}.

\section{Ingredients for the proof of the Theorem \ref{MAIN bddnees of 3paraproducts} }\label{sec main ingredients}
This section is devoted to assembling the essential components, such as estimates for Dunkl translations of Schwartz functions and the bilinear Dunkl–Calder\'on–Zygmund theory, which will be employed later to establish the boundedness of Dunkl paraproduct operators.
\subsection{Estimates for the Dunkl translations of Schwartz functions}\label{sec esti of trans of sch}
In this subsection we prove some decay estimates for Dunkl translations of functions in $\mathcal{S}(\mathbb{R}^d ).$
\begin{prop}\label{smooth function all decay}
Let $\phi$ be a smooth function on $\mathbb{R}^d$ such that ${\rm supp}\, \phi\subset B(0,r)$ for some $r>0$ and $\Phi\in \mathcal{S}(\mathbb{R}^d)$ be such that $\mathcal{F}_k\Phi=\phi$. Let $L>3d_k$ and $x, y, y'\in \mathbb{R}^d$ be such that $|y-y'|\leq 1$. Then the following hold:
   \begin{enumerate}[label=(\roman*)]
   \item \label{size decay}
$|\tau^k_x \Phi(-y)|\leq C_{r,L}\, \mu_k(B(x,1))^{-1}(1+d_G(x,y))^{-3L}(1+|x-y|)^{-1},$
\item \label{difference decay}
$|\tau^k_x \Phi(-y)- \tau^k_x \Phi(-y')|\leq C_{r,L}\, |y-y'|\,\mu_k(B(x,1))^{-1}(1+d_G(x,y))^{-3L}(1+|x-y|)^{-1}.$
   \end{enumerate}
\end{prop}
\begin{proof}
Proof of \ref{size decay} is already known (see \cite[eq. (4.7)]{HejnaRODTONRK}).\\
To prove \ref{difference decay}, we can directly use \cite[Theorem 4.7]{HejnaRODTONRK} together with the fact that
$$\mu_k\big(B(x,1)\big) \lesssim \mu_k\big(B(y,1)\big)\,(1 + d_G(x,y))^{d_k}.$$
However, we will give a short proof of \ref{difference decay}, using a similar method to that used in \cite{HejnaRODTONRK}. Let $h_t(x)$ be the Dunkl heat kernel.
Define $F(\xi)=\phi(\xi)\,e^{|\xi|^2}$. Then $F$ is smooth and ${\rm supp}\, F\subset B(0,r)$ and
$$\tau^k_x \Phi(-y)=\int_{\mathbb{R}^d}\tau^k_x(\mathcal{F}_k^{-1}F)(-z)\,h_1(x, y)\,d\mu_k(z).$$
Now using \ref{size decay} for the function $F$, \cite[Lemma 2.3]{HejnaRODTONRK}, the condition $|y-y'|\leq 1$, and then \cite[Theorem 3.1]{HejnaROADFTHSHITRDS}, we get
\begin{eqnarray*}
  &&(1+d_G(x,y))^{3L}(1+|x-y|)\,|\tau^k_x \Phi(-y)- \tau^k_x \Phi(-y')|\\
  &\leq& C_{r,L}\, \frac{(1+d_G(x,y))^{3L}(1+|x-y|)}{\mu_k(B(x,1))}\!\int_{\mathbb{R}^d}\! \frac{|h_1(x, y)-h_1(x, y')|}{(1+d_G(x,z))^{3L}(1+|x-z|)}\,d\mu_k(z)\\
  &\leq& C_{r,L}\, \frac{|y-y'|}{\mu_k(B(x,1))}\int_{\mathbb{R}^d}\frac{(1+d_G(x,z))^{3L}(1+|x-z|)(1+d_G(y,z))^{3L}(1+|y-z|)}{(1+d_G(x,z))^{3L}(1+|x-z|)}\\
  &&\times (h_2(z, y)+h_2(z, y'))\,d\mu_k(z)\\
  &\leq& C_{r,L}\, \frac{|y-y'|}{\mu_k(B(x,1))}\Big[\int_{\mathbb{R}^d}(1+d_G(y,z))^{3L}(1+|y-z|)\, h_2(z, y) \,d\mu_k(z)\\
  && +\, \int_{\mathbb{R}^d}(1+d_G(y',z))^{3L}(1+|y'-z|)\, h_2(z, y') \,d\mu_k(z)\Big]\\
  &\leq& C_{r,L}\, \frac{|y-y'|}{\mu_k(B(x,1))}\Big[\int_{\mathbb{R}^d}(1+d_G(y,z))^{3L}\, \frac{\exp\big({-c\,d_G(y,z)^2}\big)}{\mu_k(B(z,1))}\,d\mu_k(z)\\
  && + \int_{\mathbb{R}^d}(1+d_G(y',z))^{3L}\, \frac{\exp\big({-c\,d_G(y',z)^2}\big)}{\mu_k(B(z,1))}\,d\mu_k(z)\Big] \\
   &\leq& C_{r,L}\, \frac{|y-y'|}{\mu_k(B(x,1))},
\end{eqnarray*}
where in the last inequality, we have used the fact that
$$\int_{\mathbb{R}^d}\frac{\exp\big({-c\,d_G(y,z)^2}\big)}{\mu_k(B(z,1))}\,d\mu_k(z) \text{ and }\int_{\mathbb{R}^d}\frac{\exp\big({-c\,d_G(y',z)^2}\big)}{\mu_k(B(z,1))}\,d\mu_k(z)$$  are finite and independent of $y$ and $y'$ respectively.
\end{proof}

\subsection{Almost orthogonality type estimates}\label{sec almost ortho esti section}
In the next  proposition, we prove bilinear almost orthogonality type estimates involving the Dunkl metric $d_G$. The main idea of these estimates is based on the result in the classical case by Grafakos and Torres \cite[Proposition 3]{GrafakosDDFBOAADC}.
\begin{prop}\label{almost orhto esti thm}
Let $x, y_1, y_2\in \mathbb{R}^d$, $j\in \mathbb{Z}$ and $L>3d_k$. Then there exists $C>0$ such that
\begin{eqnarray*}
&&\int_{\mathbb{R}^d}\frac{d\mu_k(u)}{\big[(1+2^jd_G(x, u))(1+2^jd_G(y_1, u))(1+2^jd_G(y_2, u))\big]^{3L}}\\
&\leq& \frac{C\,\mu_k(B(x, 2^{-j})}{\big[(1+2^jd_G(x, y_1))(1+2^jd_G(x, y_2))\big]^{L}},
\end{eqnarray*}
where the constant $C$ is independent of $x,y_1,y_2$ and $j$.
\end{prop}

\begin{proof}
	Using the following simple inequalities
	$$
	(1+2^jd_G(x, u))(1+2^jd_G(y_1, u))\ge (1+2^jd_G(x, y_1))
	$$
	and
	$$
	(1+2^jd_G(y_2, u))(1+2^jd_G(x, u))\ge (1+2^jd_G(x, y_2)),
	$$
	we have 
	$$
	\begin{aligned}
		\big[(1+2^jd_G(x, u))&(1+2^jd_G(y_1, u))(1+2^jd_G(y_2, u))\big]^{3L}\\
		&\ge  (1+2^jd_G(x, u))^L (1+2^jd_G(x, y_1))^L(1+2^jd_G(x, y_2))^L.
	\end{aligned}
	$$
	It then follows that 
	$$
	\begin{aligned}
		\int_{\mathbb{R}^d}&\frac{d\mu_k(u)}{\big[(1+2^jd_G(x, u))(1+2^jd_G(y_1, u))(1+2^jd_G(y_2, u))\big]^{3L}}\\
		&\le \frac{1}{(1+2^jd_G(x, y_1))^L(1+2^jd_G(x, y_2))^L}\int_{\mathbb R^d}\frac{d\mu_k(u)}{(1+2^jd_G(x, u))^L}\\
		&\lesssim \frac{C\,\mu_k(B(x, 2^{-j})}{\big[(1+2^jd_G(x, y_1))(1+2^jd_G(x, y_2))\big]^{L}},
	\end{aligned}
	$$
	as desired.
	
\end{proof}

\subsection{Bilinear Dunkl setup and Bilinear Dunkl-Calder\'on-Zygmund theory}\label{sec multi Dunkl-Calder\'on-Zygmund theory}
We discuss a bilinear Dunkl set up as introduced in \cite{SumanWIFMFOIDS, SumanWBMTIDSVSI}. Let us consider  the root system $R$ and the multiplicity function $k$ as in Section \ref{sec prel}. Then
$$R^2:=(R\times(0)_{d})\cup \big((0)_d\times R\big),$$  is a root system in $\mathbb{R}^{2d} $. Also the reflection group associated to $R^2$ is isomorphic to the product $G\times G$. Define $k^2 : R^2\rightarrow \mathbb{C}$ by $$k^2((\lambda ,0))=k(\lambda) \text{ and }k^2((0,\lambda))=k(\lambda)  \text{ for any } \lambda \in R.$$ 
Thus we get a non-negative normalized multiplicity function $k^2$ is  on $R^2$.
Now this choice of the Root system and the multiplicity function allows us to write down the Dunkl objects on $\mathbb{R}^{2d}$ into product of the corresponding objects in $\mathbb{R}^d$.
More precisely, following the notations in Section \ref{sec prel}, for any $x_1,y_1,x_2,y_2\in \mathbb{R}^d$, 
 we have $$d\mu_{k^2}\big((x_1,x_2)\big)=d\mu_k(x_1)d\mu_k(x_2).$$  
$$\text{ and }E_{k^2}\big((x_1,x_2),(y_1,y_2)\big)=E_k(x_1,y_1)E_k(x_2,y_2).$$ 
Also for two reasonable functions $f, g$, we can write
$$\mathcal{F}_{k^2}\left(f \otimes g\right)\big((x_1,x_2)\big)=\mathcal{F}_kf(x_1)\mathcal{F}_kg(x_2)$$
$$\text{and } {\tau}^{k^2}_{(x_1,x_2)}\left(f  \otimes g\right)\big((y_1,y_2)\big)=\tau_{x_1}^kf(y_1)\,\tau_{x_2}^kg(y_2).$$
Moreover, for all $x, y_1, y_2\in \mathbb{R}^d$ and $r, r_1, r_2>0,$ the following relations hold
\begin{align}\label{comparison btween linear and bilinear objects}
\begin{split}
\left\{
\begin{aligned}
\mu_{k^2}\big(B((y_1,y_2), r)\big)\sim
\mu_k\big(B(y_1,r)\big)\,\mu_k\big(B(y_2,r)\big),\\
d_{G\times G}\big((x,x),\, (y_1,y_2)\big)\sim d_G(x,y_1)+d_G(x,y_2)\\
\text{ and }\mu_{k}\big(B(x, r_1+r_2)\big)\geq C\,[\mu_{k}\big(B(x, r_1)\big)+\mu_{k}\big(B(x, r_2)\big)].
\end{aligned}
\right.
\end{split}
\end{align}
Also the bilinear counterparts of all the properties mentioned in Section \ref{sec prel} hold in this case.
\par We now present the definition of bilinear Dunkl–Calder\'on–Zygmund type operators, as introduced in \cite{SumanWBMTIDSVSI}, which will be used repeatedly throughout this article.
\begin{defn}\label{defn of bilinear Dunkl cal zyg ope}
 A bilinear \emph{Dunkl--Calder\'on--Zygmund operator} is a function $\mathcal{T} :\mathcal{S}(\mathbb{R}^d)\times \mathcal{S}(\mathbb{R}^d)\to\mathcal{S}'(\mathbb{R}^d)$ such that for $f, g\in C_c^{\infty}(\mathbb{R}^d)$ with $\sigma (x)\notin  {\rm supp}\,f \cap  {\rm supp}\,g$ for all  $\sigma\in G$, $\mathcal{T}$ can be represented as 
  $$\mathcal{T}(f, g)(x)=\int_{\mathbb{R}^{2d}}K(x,y_1,y_2) f(y_1)g(y_2)\,d\mu_k(y_1)d\mu_k(y_2),$$
  where $K$ is a function defined away from the set $\mathcal{O}(\bigtriangleup _{3})$
  $$:=\big\{(x,y_1,y_2,)\in \mathbb{R}^{3d}:x= \sigma_1 (y_1)=\sigma_2 (y_2)\text{, for some }\sigma_1, \sigma_2 \in G\big\}$$
  and for any $x, x', y_1, y_2\in \mathbb{R}^d$, $K$ satisfies the following regularity conditions for some $0<\epsilon \leq 1$:
  \begin{eqnarray}\label{multiplier size estimate last}
&&|K(x,y_1,y_2)|\\
&\lesssim & \Big[\mu_k\big(B(x, d_G(x,y_1))\big)+\mu_k\big(B(x, d_G(x,y_2))\big)\Big]^{-2} \Big[\frac{d_G(x,y_1)+d_G(x,y_2)}{|x-y_1|+|x-y_2|}\Big]^\epsilon\nonumber
\end{eqnarray}
 for $d_G(x,y_1)+ d_G(x,y_2)>0;$
\begin{eqnarray}\label{multiplier smtness estimate last y2 changing}
    && |K(x,y_1,y_2)-K(x,y_1,y'_2)|\\
&\lesssim & \Big[\mu_k\big(B(x, d_G(x,y_1))\big)+\mu_k\big(B(x, d_G(x,y_2))\big)\Big]^{-2} \Big[\frac{|y_2-y'_2|}{\max\{|x-y_1|,\,|x-y_2|\}}\Big]^\epsilon\nonumber
\end{eqnarray}
for $|y_2-y'_2|<\max\{d_G(x,y_1)/2,\, d_G(x,y_2)/2\}$;
\begin{eqnarray}\label{multiplier smtness estimate last y1 changing}
    && |K(x,y_1,y_2)-K(x,y'_1,y_2)|\\
&\lesssim & \Big[\mu_k\big(B(x, d_G(x,y_1))\big)+\mu_k\big(B(x, d_G(x,y_2))\big)\Big]^{-2}\Big[ \frac{|y_1-y'_1|}{\max\{|x-y_1|,\,|x-y_2|\}}\Big]^\epsilon\nonumber
\end{eqnarray}
for $|y_1-y'_1|<\max\{d_G(x,y_1)/2,\, d_G(x,y_2)/2\}$\\
and 
\begin{eqnarray}\label{multiplier smtness estimate last x changing}
 && |K(x,y_1,y_2)-K(x',y_1,y_2)|\\
&\lesssim & \Big[\mu_k\big(B(x, d_G(x,y_1))\big)+\mu_k\big(B(x, d_G(x,y_2))\big)\Big]^{-2} \Big[\frac{|x-x'|}{\max\{|x-y_1|,\,|x-y_2|\}}\Big]^\epsilon\nonumber
\end{eqnarray}
for $|x-x'|<\max\{d_G(x,y_1)/2,\, d_G(x,y_2)/2\}$.
\end{defn}
 The boundedness of $m$-linear Dunkl–Calder\'on–Zygmund type operators with $m$-fold $G$-invariant Muckenhoupt weights was established in \cite{SumanWBMTIDSVSI}. In this article, we focus on the bilinear case without weights. For clarity, we state the corresponding boundedness result in the theorem below.
\begin{thm}\label{bilinear calderon Zygmund bddness thm}\cite[Theorem 5.3]{SumanWBMTIDSVSI}
    Let $1< p_1,p_2<\infty$, $p$ be the number given by $1/p=1/p_1+1/p_2$, and  $\mathcal{T}$ maps from $L^{q_1}(\mathbb{R}^d,\,d\mu_k)\times L^{q_2}(\mathbb{R}^d,\,d\mu_k)$ to $L^{q,\,\infty}(\mathbb{R}^d,\, d\mu_k)$ for some $q,q_1,q_2$ satisfying $1\leq q_1,q_2<\infty$ with $1/q=1/q_1+1/q_2$. Then the bilinear operator $\mathcal{T}$ is bounded from $L^{p_1}(\mathbb{R}^d,\,d\mu_k)\times L^{p_2}(\mathbb{R}^d,\,d\mu_k)$ to $L^{p}(\mathbb{R}^d,\, d\mu_k)$.
\end{thm}

\section{Boundedness of Dunkl Paraproducts}\label{sec Dunkl paraproduct}
In this section, we study the boundedness of Dunkl paraproduct operators on products of Lebesgue spaces. In fact, the following proposition establishes boundedness results for a particular class of Dunkl paraproduct operators.
\begin{prop}\label{bddness of general 3paraproducts}
  For any $\psi, \phi, \theta \in C_c^{\infty}(\mathbb{R}^d)$, if the supports of at least two among $\psi$, $\phi$, and $\theta$ do not contain $0$, then the Dunkl paraproduct operator $\Pi[\theta, \psi, \phi]$ is bounded from $L^{q_1}(\mathbb{R}^d,\,d\mu_k)\times L^{q_2}(\mathbb{R}^d,\,d\mu_k)$ to $L^{q}(\mathbb{R}^d,\, d\mu_k)$ for all $q,q_1,q_2$ satisfying $1< q, q_1,q_2<\infty$ with $1/q=1/q_1+1/q_2$. 
\end{prop}
\begin{proof}
Here we take
 $${\rm supp}\, \psi \subset \{\xi \in \mathbb{R}^d: 1/r_1\leq |\xi|\leq r_1\}$$ 
 $$\text{and } {\rm supp}\, \theta \subset \{\xi \in \mathbb{R}^d: 1/r_2\leq |\xi|\leq r_2\},$$ for some $r_1, r_2>1$ and consider $h$ in $\mathcal{S}(\mathbb{R}^d)$ such that $||h||_{L^{q'}(d\mu_k)}$=1. Then, by repeatedly applying the Plancherel formula for the Dunkl transform, using the Cauchy--Schwarz inequality for the $\ell^2$ norm, H\"older's inequality, and finally the Littlewood--Paley theory for the Dunkl transform \cite[Theorem 3.1 and Theorem 3.4]{SumanWBMTIDSVSI}, we can easily write

  \begin{eqnarray*}
   &&\big| \int_{\mathbb{R}^d}  \Pi[\theta, \psi, \phi] (fg)(x)\, h(x) \,d\mu_k(x) \big|\\
   &=& \big| \sum\limits_{j\in \mathbb{Z}}\int_{\mathbb{R}^d} \mathcal{F}_k^{-1} \big( \theta_j \, \mathcal{F}_k((\Psi_j*_kf)(\cdot)(\Phi_j*_kg)(\cdot))\big)(x) \, h(x) \,d\mu_k(x) \big|\\
   &=& \big| \sum\limits_{j\in \mathbb{Z}}\int_{\mathbb{R}^d}  (\Psi_j*_kf)(x)\,(\Phi_j*_kg)(x) \, (\Theta_j*_kh)(x) \,d\mu_k(x) \big|\\
   &\leq & \int_{\mathbb{R}^d} \big(\sum\limits_{j\in \mathbb{Z}} | (\Psi_j*_kf)(x)\,(\Phi_j*_kg)(x)|^2\big)^{1/2} \big(\sum\limits_{j\in \mathbb{Z}} |(\Theta_j*_kh)(x)|^2\big)^{1/2}\, d\mu_k(x)\\
   &\leq & \big\|\sum\limits_{j\in \mathbb{Z}} | (\Psi_j*_kf)(\cdot)\,(\Phi_j*_kg)(\cdot)|^2\big)^{1/2}\big\|_{L^{q}(d\mu_k)} \big\| \sum\limits_{j\in \mathbb{Z}} |(\Theta_j*_kh)|^2\big)^{1/2}\big\|_{L^{q'}(d\mu_k)}\\
   &\leq & \big\|\sum\limits_{j\in \mathbb{Z}} | (\Psi_j*_kf)|^2\big)^{1/2}\big\|_{L^{q_1}(d\mu_k)} \big\| \sup\limits_{j\in \mathbb{Z}} |(\Phi_j*_kg)|\big\|_{L^{q_2}(d\mu_k)}\sum\limits_{j\in \mathbb{Z}} |(\Theta_j*_kh)|^2\big)^{1/2}\big\|_{L^{q'}(d\mu_k)}\\
   &\leq & C\, ||f||_{L^{q_1}(d\mu_k)}||g||_{L^{q_2}(d\mu_k)}||h||_{L^{q'}(d\mu_k)}.
  \end{eqnarray*}
 For the remaining cases, the necessary routine adjustments are straightforward and left to the reader.
\end{proof}
The previous proposition does not include the boundedness of the operators for the full range $1/2<q<\infty$. However, using the bilinear Dunkl--Calder\'on--Zygmund theory (Theorem \ref{bilinear calderon Zygmund bddness thm}), we establish the boundedness of these Dunkl paraproduct operators over the full range, completing the proof of Theorem~\ref{MAIN bddnees of 3paraproducts}.

\begin{proof}[{\bf Proof of Theorem \ref{MAIN bddnees of 3paraproducts}}]
We begin by noting that the operator $\Pi[\theta, \psi, \phi]$ admits an integral representation of the form
\begin{equation*}
 \Pi[\theta, \psi, \phi] (f,g)(x) =\int_{\mathbb{R}^{2d}} K(x, y_1, y_2) f(y_1)g(y_2) \, d\mu_k(y_1)\,d\mu_k(y_2),
\end{equation*}
where $K$ is given by
$$ K(x,y_1,y_2)= \sum\limits_{j\in \mathbb{Z}}\int_{\mathbb{R}^d}\tau^k_{-u}\Theta_j(x) \tau^k_u \Psi_j(-y_1) \tau^k_u \Phi_j(-y_2)\, d\mu_k(u).$$
Thus, since Proposition \ref{bddness of general 3paraproducts} has been proved, to apply Theorem \ref{bilinear calderon Zygmund bddness thm}, it remains only to verify the conditions \eqref{multiplier size estimate last}, \eqref{multiplier smtness estimate last y2 changing}, \eqref{multiplier smtness estimate last y1 changing}, and \eqref{multiplier smtness estimate last x changing} for the kernel $K$ above.
\par Let us verify the condition \eqref{multiplier size estimate last} first. Take $d_G(x,y_1)+d_G(x,y_2)>0$, then from Proposition \ref{smooth function all decay}\ref{size decay}, Proposition \ref{almost orhto esti thm}, and \eqref{Volofball}, we have 
\begin{eqnarray*}
&& | K(x,y_1,y_2)| \\
&\leq & C\, \sum\limits_{j\in \mathbb{Z}}2^{3jd_k}\int_{\mathbb{R}^d} \mu_k(B(2^jx,1))^{-1}(1+2^jd_G(u,x))^{-3L}(1+2^j|u-x|)^{-1}\\
 &&\times\mu_k(B(2^jy_1,1))^{-1}(1+2^jd_G(u,y_1))^{-3L}(1+2^j|u-y_1|)^{-1}\\
 &&\times\mu_k(B(2^jy_2,1))^{-1}(1+2^jd_G(u,y_2))^{-3L}(1+2^j|u-y_2|)^{-1}\, d\mu_k(u)\\
 &\leq & C\, \sum\limits_{j\in \mathbb{Z}}\frac{2^{3jd_k}}{\mu_k(B(2^jx,1))\mu_k(B(2^jy_1,1))\mu_k(B(2^jy_2,1))}\\
 &&\times\int_{\mathbb{R}^d}\frac{\{1+(2^j|u-x|+2^j|u-y_1|)+ (2^j|u-x|+2^j|u-y_2|)\}^{-1}d\mu_k(u)}{\big[(1+2^jd_G(x, u))(1+2^jd_G(y_1, u))(1+2^jd_G(y_2, u))\big]^{3L}}\\
 &\leq & C\, \sum\limits_{j\in \mathbb{Z}}\frac{1}{2^j(|x-y_1|+|x-y_2|)}\frac{2^{3jd_k}}{\mu_k(B(2^jx,1))\mu_k(B(2^jy_1,1))\mu_k(B(2^jy_2,1))}\\
 &&\times\int_{\mathbb{R}^d}\frac{d\mu_k(u)}{\big[(1+2^jd_G(x, u))(1+2^jd_G(y_1, u))(1+2^jd_G(y_2, u))\big]^{3L}}\\
&\leq & C\, \sum\limits_{j\in \mathbb{Z}}\frac{1}{2^j(|x-y_1|+|x-y_2|)}\frac{2^{3jd_k}}{\mu_k(B(2^jx,1))\mu_k(B(2^jy_1,1))\mu_k(B(2^jy_2,1))}\\
 &&\times \frac{\mu_k(B(x, 2^{-j})}{\big[1+2^j(d_G(x, y_1))+ d_G(x, y_2))\big]^{L}}\\
&\leq & C\, \sum\limits_{j\in \mathbb{Z}}\frac{1}{2^j(|x-y_1|+|x-y_2|)}\frac{2^{2jd_k}}{\mu_k(B(2^jx,1))\mu_k(B(2^jy_1,1))\mu_k(B(2^jy_2,1))}\\
 &&\times \frac{\mu_k(B(2^jx, 1)}{\big[1+2^j(d_G(x, y_1))+ d_G(x, y_2))\big]^{L}}\\
&\leq & C\,\frac{1}{|x-y_1|+|x-y_2|}\sum\limits_{j\in \mathbb{Z}}\frac{2^{-j}2^{2jd_k}}{\big[1+2^j(d_G(x, y_1))+ d_G(x, y_2))\big]^{L}}\\
&&\times\frac{1}{\mu_k(B(2^jy_1,1))\mu_k(B(2^jy_2,1))}.
\end{eqnarray*}
Now, to estimate the sum above, we use \eqref{Volofball} again, use the relations \eqref{comparison btween linear and bilinear objects}, and break it into two parts as follows\\
\begin{eqnarray*}
    &&\sum\limits_{j\in \mathbb{Z}}\frac{2^{-j}2^{2jd_k}}{\big[1+2^j(d_G(x, y_1))+ d_G(x, y_2))\big]^{L}}\frac{1}{\mu_k(B(2^jy_1,1))\mu_k(B(2^jy_2,1))}\\
    &=& \sum\limits_{j\in \mathbb{Z}}\frac{2^{-j}}{\big[1+2^j(d_G(x, y_1))+ d_G(x, y_2))\big]^{L}}\frac{1}{\mu_k(B(y_1,2^{-j}))\mu_k(B(y_2,2^{-j}))}\\
     &\leq& C\,\sum\limits_{j\in \mathbb{Z}}\frac{2^{-j}}{\big[1+2^j(d_G(x, y_1))+ d_G(x, y_2))\big]^{L}}
    \frac{1}{\mu_{k^2}(B((y_1,y_2),2^{-j}))}\\
    &=&C\,\sum\limits_{j\in \mathbb{Z}: 2^j(d_G(x, y_1))+ d_G(x, y_2))\leq 1} \ldots+ \sum\limits_{j\in \mathbb{Z}: 2^j(d_G(x, y_1))+ d_G(x, y_2))>1}\ldots .
\end{eqnarray*}
For the first part, we use \eqref{VOLRADREL} in $\mathbb{R}^{2d}$, \eqref{comparison btween linear and bilinear objects}, \eqref{V(x,y,d(x,y) comparison} in $\mathbb{R}^{2d}$, and \eqref{comparison btween linear and bilinear objects} repeatedly so that 
\begin{eqnarray*}
  &&\sum\limits_{j\in \mathbb{Z}: 2^j(d_G(x, y_1))+ d_G(x, y_2))\leq 1}  \frac{2^{-j}}{\big[1+2^j(d_G(x, y_1))+ d_G(x, y_2))\big]^{L}}
    \frac{1}{\mu_{k^2}(B((y_1,y_2),2^{-j}))}\\
    &\leq &\frac{C}{\mu_{k^2}(B((y_1,y_2),d_G(x, y_1))+ d_G(x, y_2)))}\sum\limits_{j\in \mathbb{Z}: 2^j(d_G(x, y_1))+ d_G(x, y_2))\leq 1} \\
    &&\times\frac{2^{-j}}{\big[1+2^j(d_G(x, y_1))+ d_G(x, y_2))\big]^{L}} \left[\frac{d_G(x, y_1))+ d_G(x, y_2))}{2^{-j}}\right]^{2d} \\
    &\leq & C\frac{d_G(x,y_1)+ d_G(x,y_2)}{\mu_{k^2}(B((y_1,y_2),d_G(x, y_1))+ d_G(x, y_2)))}\sum\limits_{j\in \mathbb{Z}: 2^j(d_G(x, y_1))+ d_G(x, y_2))\leq 1}\\
    &&\times\big(2^j(d_G(x, y_1))+ d_G(x, y_2))\big)^{2d-1}\\
    &\leq & C\frac{d_G(x,y_1)+ d_G(x,y_2)}{\mu_{k^2}(B((y_1,y_2),d_G(x, y_1))+ d_G(x, y_2)))}\\
    &\leq &C \frac{d_G(x,y_1)+ d_G(x,y_2)}{\mu_{k^2}(B((y_1,y_2),d_{G\times G}((x,x)), (y_1, y_2)))}\\
    &\leq & C \frac{d_G(x,y_1)+ d_G(x,y_2)}{\mu_{k^2}(B((x,x),d_{G\times G}((x,x)), (y_1, y_2)))}\\
    &\leq & C\, \big(d_G(x,y_1)+ d_G(x,y_2)\big)\Big[\mu_k\big(B(x, d_G(x,y_1))\big)+\mu_k\big(B(x, d_G(x,y_2))\big)\Big]^{-2}
\end{eqnarray*}
For the second part, we proceed similarly and obtain
\begin{eqnarray*}
  &&\sum\limits_{j\in \mathbb{Z}: 2^j(d_G(x, y_1))+ d_G(x, y_2))> 1}  \frac{2^{-j}}{\big[1+2^j(d_G(x, y_1))+ d_G(x, y_2))\big]^{L}}
    \frac{1}{\mu_{k^2}(B((y_1,y_2),2^{-j}))}\\
    &\leq &\frac{C}{\mu_{k^2}(B((y_1,y_2),d_G(x, y_1))+ d_G(x, y_2)))}\sum\limits_{j\in \mathbb{Z}: 2^j(d_G(x, y_1))+ d_G(x, y_2))> 1} \\
    &&\times\frac{2^{-j}}{\big[1+2^j(d_G(x, y_1))+ d_G(x, y_2))\big]^{L}} \left[\frac{d_G(x, y_1))+ d_G(x, y_2))}{2^{-j}}\right]^{2d_k} \\ 
    &\leq & C\frac{d_G(x,y_1)+ d_G(x,y_2)}{\mu_{k^2}(B((y_1,y_2),d_G(x, y_1))+ d_G(x, y_2)))}\sum\limits_{j\in \mathbb{Z}: 2^j(d_G(x, y_1))+ d_G(x, y_2))> 1}\\
    &&\times\big(2^j(d_G(x, y_1))+ d_G(x, y_2))\big)^{2d_k-L-1}.\\
\end{eqnarray*}
Hence, in the same way as for the first part, we can show that the above term is dominated by
$$\big( d_G(x,y_1)+ d_G(x,y_2)\big)\big[\mu_k\big(B(x, d_G(x,y_1))\big)+\mu_k\big(B(x, d_G(x,y_2))\big)\big]^{-2}.$$ 
Thus, the verification of condition \eqref{multiplier size estimate last} is complete.
\par Next, we verify the condition \eqref{multiplier smtness estimate last y2 changing}. Let $|y_2 - y'_2| < \max\{d_G(x, y_1)/2,\, d_G(x, y_2)/2\}$ and consider the difference
\begin{eqnarray*}
   &&|K(x,y_1,y_2)-K(x,y_1,y'_2)|\\
   &\leq& \sum\limits_{j\in \mathbb{Z}}\int_{\mathbb{R}^d}|\tau^k_{-u}\Theta_j(x)|\,| \tau^k_u \Psi_j(-y_1) |\,|\tau^k_u \Phi_j(-y_2)-\tau^k_u \Phi_j(-y'_2)|\, d\mu_k(u).
\end{eqnarray*}
When $|2^jy_2-2^jy'_2|\leq 1$, by Proposition \ref{smooth function all decay}\ref{size decay} and \ref{difference decay}, we obtain
\begin{eqnarray*}
   &&\int_{\mathbb{R}^d}|\tau^k_{-u}\Theta_j(x)|\,| \tau^k_u \Psi_j(-y_1) |\,|\tau^k_u \Phi_j(-y_2)-\tau^k_u \Phi_j(-y'_2)|\, d\mu_k(u)\\
   &\leq & C\, 2^{3jd_k}2^j|y_2-y'_2|\int_{\mathbb{R}^d} \mu_k(B(2^jx,1))^{-1}(1+2^jd_G(u,x))^{-3L}(1+2^j|u-x|)^{-1}\\
 &&\times\mu_k(B(2^jy_1,1))^{-1}(1+2^jd_G(u,y_1))^{-3L}(1+2^j|u-y_1|)^{-1}\\
 &&\times\mu_k(B(2^jy_2,1))^{-1}(1+2^jd_G(u,y_2))^{-3L}(1+2^j|u-y_2|)^{-1}\, d\mu_k(u)\\
 &\leq& C\, \big[\mu_k\big(B(x, d_G(x,y_1))\big)+\mu_k\big(B(x, d_G(x,y_2))\big)\big]^{-2}\frac{|y_2-y'_2|}{|x-y_1|+|x-y_2|}\\
 &&\times \left\{ \frac{\big(2^j(d_G(x, y_1))+ d_G(x, y_2))\big)^{2d} + \big(2^j(d_G(x, y_1))+ d_G(x, y_2))\big)^{2d_k}}{\big[1+2^j(d_G(x, y_1))+ d_G(x, y_2))\big]^{L}} \right\},
   \end{eqnarray*}
where the last step is obtained by performing calculations similar to those carried out in the verification of condition \eqref{multiplier size estimate last}.

\par On the other hand, when $|2^jy_2-2^jy'_2| > 1$, Proposition \ref{smooth function all decay}\ref{size decay} and the condition $1<|2^jy_2-2^jy'_2|$ yields
\begin{eqnarray*}
   &&\int_{\mathbb{R}^d}|\tau^k_{-u}\Theta_j(x)|\,| \tau^k_u \Psi_j(-y_1) |\,(|\tau^k_u \Phi_j(-y_2)|+|\tau^k_u \Phi_j(-y'_2)|)\, d\mu_k(u)\\
   &\leq & C\, |2^jy_2-2^jy'_2|\big[2^{3jd_k}\int_{\mathbb{R}^d} \mu_k(B(2^jx,1))^{-1}(1+2^jd_G(u,x))^{-3L}(1+2^j|u-x|)^{-1}\\
 &&\times\mu_k(B(2^jy_1,1))^{-1}(1+2^jd_G(u,y_1))^{-3L}(1+2^j|u-y_1|)^{-1}\\
 &&\times\mu_k(B(2^jy_2,1))^{-1}(1+2^jd_G(u,y_2))^{-3L}(1+2^j|u-y_2|)^{-1}\, d\mu_k(u)\\
 &&+ 2^{3jd_k}\int_{\mathbb{R}^d} \mu_k(B(2^jx,1))^{-1}(1+2^jd_G(u,x))^{-3L}(1+2^j|u-x|)^{-1}\\
 &&\times\mu_k(B(2^jy_1,1))^{-1}(1+2^jd_G(u,y_1))^{-3L}(1+2^j|u-y_1|)^{-1}\\
 &&\times\mu_k(B(2^jy'_2,1))^{-1}(1+2^jd_G(u,y'_2))^{-3L}(1+2^j|u-y'_2|)^{-1}\, d\mu_k(u)\big]\\
 &\leq & \big[\mu_k\big(B(x, d_G(x,y_1))\big)+\mu_k\big(B(x, d_G(x,y_2))\big)\big]^{-2}\frac{|y_2-y'_2|}{|x-y_1|+|x-y_2|}\\
 &&\times \left\{ \frac{\big(2^j(d_G(x, y_1))+ d_G(x, y_2))\big)^{2d} + \big(2^j(d_G(x, y_1))+ d_G(x, y_2))\big)^{2d_k}}{\big[1+2^j(d_G(x, y_1))+ d_G(x, y_2))\big]^{L}} \right\}\\
 && + \big[\mu_k\big(B(x, d_G(x,y_1))\big)+\mu_k\big(B(x, d_G(x,y'_2))\big)\big]^{-2}\frac{|y_2-y'_2|}{|x-y_1|+|x-y'_2|}\\
 &&\times \left\{ \frac{\big(2^j(d_G(x, y_1))+ d_G(x, y'_2))\big)^{2d} + \big(2^j(d_G(x, y_1))+ d_G(x, y'_2))\big)^{2d_k}}{\big[1+2^j(d_G(x, y_1))+ d_G(x, y'_2))\big]^{L}} \right\},
\end{eqnarray*}
where the last step is derived in the same way as in the case $|2^j y_2 - 2^j y'_2| \leq 1$.
\par Now, the bilinear setup introduced in Section \ref{sec multi Dunkl-Calder\'on-Zygmund theory} allows us to write
$$d_G(x, y_1)\leq d_{G\times G}((x,x),\, (y_1,y_2)) \text{ and }d_G(x, y_2)\leq d_{G\times G}((x,x),\, (y_1,y_2)),$$
 and therefore the condition $|y_2 - y'_2| < \max\{d_G(x, y_1)/2,\, d_G(x, y_2)/2\}$ turns into the condition 
\begin{eqnarray*}
|y_2 - y'_2|&\leq& d_{G\times G}((x,x),\, (y_1,y_2))/2.
\end{eqnarray*}
We can now show that the above condition implies that
$$d_G(x, y_1))+ d_G(x, y'_2)\sim d_G(x, y_1))+ d_G(x, y_2),$$
$$ (|x-y_1| + |x-y'_2|\sim (|x-y_1| + |x-y_2|,$$
$$\text{ and }\mu_k\big(B(x, d_G(x,y_1))\big)+\mu_k\big(B(x, d_G(x,y'_2))\big)$$
$$\sim \mu_k\big(B(x, d_G(x,y_1))\big)+\mu_k\big(B(x, d_G(x,y_2))\big).$$
\par So, finally, when $|2^j y_2 - 2^j y'_2| > 1$, keeping the above estimate in mind, we can write
\begin{eqnarray*}
   &&\int_{\mathbb{R}^d}|\tau^k_{-u}\Theta_j(x)|\,| \tau^k_u \Psi_j(-y_1) |\,|\tau^k_u \Phi_j(-y_2)-\tau^k_u \Phi_j(-y'_2)|\, d\mu_k(u)\\
   &\leq& C\, \big[\mu_k\big(B(x, d_G(x,y_1))\big)+\mu_k\big(B(x, d_G(x,y_2))\big)\big]^{-2}\frac{|y_2-y'_2|}{|x-y_1|+|x-y_2|}\\
 &&\times \left\{ \frac{\big(2^j(d_G(x, y_1))+ d_G(x, y_2))\big)^{2d} + \big(2^j(d_G(x, y_1))+ d_G(x, y_2))\big)^{2d_k}}{\big[1+2^j(d_G(x, y_1))+ d_G(x, y_2))\big]^{L}} \right\}.
   \end{eqnarray*}
Now, the verification of condition \eqref{multiplier smtness estimate last y2 changing} can be carried out in exactly the same manner as in the verification of condition \eqref{multiplier size estimate last}, since we have actually shown that 
\begin{eqnarray*}
   &&|K(x,y_1,y_2)-K(x,y_1,y'_2)|\\
    &\leq& C\, \big[\mu_k\big(B(x, d_G(x,y_1))\big)+\mu_k\big(B(x, d_G(x,y_2))\big)\big]^{-2}\frac{|y_2-y'_2|}{|x-y_1|+|x-y_2|}\\
 &&\times\Big\{\sum\limits_{j\in \mathbb{Z}: 2^j(d_G(x, y_1))+ d_G(x, y_2))\leq 1}\big(2^j(d_G(x, y_1))+ d_G(x, y_2))\big)^{2d}\\
&&+ \sum\limits_{j\in \mathbb{Z}: 2^j(d_G(x, y_1))+ d_G(x, y_2))> 1}\big(2^j(d_G(x, y_1))+ d_G(x, y_2))\big)^{2d_k-L}\Big\}.
 \end{eqnarray*}
 \par The verification of conditions \eqref{multiplier smtness estimate last y1 changing} and \eqref{multiplier smtness estimate last x changing} is similar to that of \eqref{multiplier smtness estimate last y2 changing}, and hence is not included in this article.
\end{proof}
\begin{rem}
It is worth noting that a more general theorem than Theorem~\ref{MAIN bddnees of 3paraproducts} follows from \cite[Theorem 5.2]{SumanWBMTIDSVSI} and \cite[Theorem 5.3]{SumanWBMTIDSVSI}, which establish one-weight and two-weight inequalities involving $G$-invariant multiple Muckenhoupt weights. However, to keep the paper concise, we do not introduce weights here and restrict our attention to the unweighted case.
\end{rem}

\section{Proofs of the fractional Leibniz rules}\label{sec frac Leib rule}
\subsection{Paraproduct decomposition}\label{sec paraproduct decompo}
We now show that the product of two Schwartz functions $f$ and $g$ can be decomposed into a sum of paraproducts, as stated in the next proposition. Before that, we present a technical lemma, a version of which is already known (see \cite{WrobelABMVAFC} or \cite{SumanWBMTIDSVSI}), but for the sake of completeness, we include a proof here. This lemma will be used in the forthcoming proposition.

\begin{lem}\label{conv support lemma}
 Let $j \in \mathbb{Z}$, and let $\psi$ and $\phi$ be Schwartz functions such that $\psi$ is supported in $\{\xi \in \mathbb{R}^d : 2^{j-1} \leq |\xi| \leq 2^{j+1} \}$, and $\phi$ is supported in $\{\xi \in \mathbb{R}^d : |\xi| \leq 2^{j-3} \}$. Then their convolution $\psi *_k \phi$ is supported in $\{\xi \in \mathbb{R}^d : 2^{j-2} \leq |\xi| \leq 2^{j+2} \}$. 
\end{lem}
\begin{proof}
From definition of of Dunkl convolution and using the support of $\psi$, we have the expression
\begin{equation*}
 \psi *_k \phi (x)= \int\limits_{2^{j-1} \leq |y| \leq 2^{j+1} }  \tau^k_x \phi(-y)\, \psi(y)\,d\mu_k(y). 
\end{equation*}

Now from \cite[Theorem 5.1]{AmriTRIDA} (see also \cite[Theorem 1.7]{HejnaHMT}), it follows that
$${\rm supp} \, \tau^k_x \phi(-\cdot) \subset \{y \in \mathbb{R}^d: |x|- 2^{j-3}\leq |y|\leq |x|+2^{j-3}\}.$$
So, if $y\in {\rm supp}\, \tau^k_x \phi(-\cdot)$ and $x \notin \{\xi \in \mathbb{R}^d : 2^{j-2} \leq |\xi| \leq 2^{j+2} \}$, then either $|y|<2^{j-1}$ or $|y|>2^{j+1}$. But this implies that the above integral is zero. Hence, the proof follows.
\end{proof}
\begin{prop}\label{decompo of fg in paraproduct}
Let $f, g\in \mathcal{S}(\mathbb{R}^d )$, then $fg$ can be written as 
$$(fg)(x)= \Pi_1(f, g)(x) + \Pi_2(f, g)(x) + \Pi_3(f, g)(x),$$
 where $\Pi_1=\Pi[\theta^{(1)} , \psi^{(1)}, \phi^{(1)}]$ for some $\psi^{(1)}, \phi^{(1)}, \theta^{(1)} \in C_c^\infty(\mathbb{R}^d)$  such that $${\rm supp}\, \psi^{(1)} \subset \{\xi \in \mathbb{R}^d: 1/2\leq |\xi|\leq 2\}, {\rm supp}\,\phi^{(1)} \subset \{\xi \in \mathbb{R}^d: 2^{-5}\leq |\xi|\leq 2^{5}\},$$
 $$ {\rm supp}\,\theta^{(1)} \subset \{\xi \in \mathbb{R}^d: |\xi|\leq 2^{7}\} \text{ with } 0\leq \theta^{(1)}\leq 1;$$
 $\Pi_2=\Pi[\theta^{(2)}, \psi^{(2)}, \phi^{(2)}]$ for some $\psi^{(2)}, \phi^{(2)}, \theta^{(2)} \in C_c^\infty(\mathbb{R}^d)$  such that $${\rm supp}\, \psi^{(2)} \subset \{\xi \in \mathbb{R}^d: 1/2\leq |\xi|\leq 2\},
{\rm supp}\,\phi^{(2)} \subset \{\xi \in \mathbb{R}^d: |\xi|\leq 2^{-3}\},$$ 
 $$ {\rm supp}\,\theta^{(2)} \subset \{\xi \in \mathbb{R}^d: 2^{-3}\leq |\xi|\leq 2^{3}\};$$
 and $\Pi_3=\Pi[\theta^{(3)}, \psi^{(3)}, \phi^{(3)}]$ for some $\psi^{(3)}, \phi^{(3)}, \theta^{(3)} \in C_c^\infty(\mathbb{R}^d)$  such that  
 $${\rm supp}\, \psi^{(3)} \subset \{\xi \in \mathbb{R}^d: |\xi|\leq 2^{-3}\}, {\rm supp}\,\phi^{(3)} \subset \{\xi \in \mathbb{R}^d: 1/2 \leq |\xi|\leq 2\},$$ 
 $$ {\rm supp}\, \theta^{(3)} \subset \{\xi \in \mathbb{R}^d: 2^{-3}\leq |\xi|\leq 2^{3}\}.$$
\end{prop}
\begin{proof}
In the proof, we will follow a strategy introduced in the book by Muscalu-Schlag \cite{MuscaluBook2CAMHA}.
\par Let $\psi\in C^{\infty}(\mathbb{R}^d)$ be such that ${\rm supp}\, \psi \subset \{\xi\in\mathbb{R}^d: 1/2\leq |\xi|\leq 2\}$ and 
$$\sum\limits_{j\in\mathbb{Z}}\psi_j(\xi)=1 \text{ for all } \xi\neq 0,$$
where $\psi_j(\xi)=\psi(\xi/2^j)$ for all $\xi\in\mathbb{R}^d.$
Then 
\begin{eqnarray*}
(fg)(x)&=&\int_{\mathbb{R}^{2d}}\sum\limits_{j_1\in\mathbb{Z}}\sum\limits_{j_2\in\mathbb{Z}} \psi_{j_1}(\xi)\psi_{j_2}(\eta) \mathcal{F}_kf(\xi)\mathcal{F}_kg(\eta)E_k(ix,\xi) E_k(ix,\eta)\,d\mu_k(\xi)d\mu_k(\eta)\\
&=&\int_{\mathbb{R}^{2d}}\sum\limits_{|j_1-j_2|\leq 4}\cdots +\int_{\mathbb{R}^{2d}}\sum\limits_{j_1>j_2+4}\cdots +\int_{\mathbb{R}^{2d}}\sum\limits_{j_2>j_1+4}\cdots\\
&=:&\Pi_1(f, g)(x) + \Pi_2(f, g)(x) + \Pi_3(f, g)(x).
\end{eqnarray*}

Now we can explicitly compute $\Pi_1(f, g)$ as
\begin{eqnarray*}
\Pi_1(f, g)(x)&=&\int_{\mathbb{R}^{2d}}\sum\limits_{|j_1-j_2|\leq 4} \psi_{j_1}(\xi)\psi_{j_2}(\eta) \mathcal{F}_kf(\xi)\mathcal{F}_kg(\eta)E_k(ix,\xi) E_k(ix,\eta)\,d\mu_k(\xi)d\mu_k(\eta)\\
&=& \int_{\mathbb{R}^{2d}}\sum\limits_{j\in \mathbb{Z}}\psi_j(\xi)\sum\limits_{j_2:|j-j_2|\leq 4}\psi_{j_2}(\eta)\mathcal{F}_kf(\xi)\mathcal{F}_kg(\eta)E_k(ix,\xi) E_k(ix,\eta)\,d\mu_k(\xi)d\mu_k(\eta)\\
&= : &\sum\limits_{j\in \mathbb{Z}} \int_{\mathbb{R}^{2d}}\psi^{(1)}_j(\xi)\, \phi^{(1)}_j(\eta)\mathcal{F}_kf(\xi)\mathcal{F}_kg(\eta)E_k(ix,\xi) E_k(ix,\eta)\,d\mu_k(\xi)d\mu_k(\eta),
\end{eqnarray*}
where $\psi^{(1)}(\xi)=\psi(\xi)$, $\psi^{(1)}_j(\xi)=\psi^{(1)}(\xi/2^{j})$, $ \phi^{(1)}(\eta)=\sum\limits_{|j|\leq 4}\psi_j(\eta)$, and $ \phi^{(1)}_j(\eta)= \phi^{(1)}(\eta/2^j)$.
\par Then clearly ${\rm supp}\,  \phi^{(1)}\subset \{\xi\in\mathbb{R}^d: 2^{-5}\leq |\xi|\leq 2^5\}$ and we can write
$$\Pi_1(f, g)(x)=\sum\limits_{j\in \mathbb{Z}} \big((\Psi^{(1)}_j*_kf)(x)(\Phi^{(1)}_j*_kg)(x)\big).$$
Now, to obtain the exact expression for $\Pi_1(f, g)$ as in the statement of the theorem, we need to work a bit more by using the properties of Dunkl convolution and the support of the functions $\psi^{(1)}$ and $\phi^{(1)}$. In fact, let us introduce a new function $\theta^{(1)} \in C_c^{\infty}(\mathbb{R}^d )$ such that $0\leq \theta^{(1)}\leq 1$, ${\rm supp}\, \theta^{(1)}\subset\{\xi\in \mathbb{R}^d:|\xi|\leq 2^{7}\}$, $\theta^{(1)}(\xi)=1$ for $|\xi|\leq 2^{6}$, and for any $j\in \mathbb{Z}$, define $\theta^{(1)}_j(\xi)=\theta^{(1)} (\xi/2^j)$. We claim that for any $j\in \mathbb{Z}$,
\begin{equation}\label{play again with conv}
(\Psi^{(1)}_j*_kf)(x)(\Phi^{(1)}_j*_kg)(x)=\Theta^{(1)}_j *_k\big((\Psi^{(1)}_j*_kf)(\cdot)(\Phi^{(1)}_j*_kg)(\cdot)\big)(x),
\end{equation}
and hence, we can write
\begin{equation*}
\Pi_1(f, g)(x)=\sum\limits_{j\in \mathbb{Z}} \Theta^{(1)}_j *_k\big((\Psi^{(1)}_j*_kf)(\cdot)(\Phi^{(1)}_j*_kg)(\cdot)\big)(x).
\end{equation*}

By uniqueness of the Dunkl transform, to prove \eqref{play again with conv}, we need to show that
$$\mathcal{F}_k\big((\Psi^{(1)}_j*_kf)(\cdot)(\Phi^{(1)}_j*_kg)(\cdot)\big)(\xi) = \mathcal{F}_k\big(\Theta^{(1)}_j *_k\big((\Psi^{(1)}_j*_kf)(\cdot)(\Phi^{(1)}_j*_kg)(\cdot)\big)\big)(\xi) $$
which, after simplification, becomes
\begin{equation}\label{play again with Dunkl conv part2}
(\psi^{(1)}_j(\cdot)\mathcal{F}_kf(\cdot)\big)*_k (\phi^{(1)}_j(\cdot)\mathcal{F}_kg(\cdot)\big)(\xi)= \theta^{(1)}_j(\xi) (\psi^{(1)}_j(\cdot)\mathcal{F}_kf(\cdot)\big)*_k (\phi^{(1)}_j(\cdot)\mathcal{F}_kg(\cdot)\big)(\xi).  
\end{equation}
Again since $\psi^{(1)}_j$ is supported in $\{\xi \in \mathbb{R}^d:  2^{j-1}\leq |\xi|\leq 2^{j+1}\}$ and $\phi^{(1)}_j$ is supported in $\{\xi \in \mathbb{R}^d: |\xi|\leq 2^{j+5}\}$, it is easy to see that (similar to the proof of Lemma \ref{conv support lemma})
$$supp\, (\psi^{(1)}_j(\cdot)\mathcal{F}_kf(\cdot)\big)*_k (\phi^{(1)}_j(\cdot)\mathcal{F}_kg(\cdot)\big) \subset \{\xi \in \mathbb{R}^d:  |\xi|\leq 2^{j+6}\}.$$
But $\theta^{(1)}_j(\xi)=1$ on $  |\xi|\leq 2^{j+6}$. This concludes the proof of \eqref{play again with Dunkl conv part2}.

\par For $\Pi_2(f, g)$, we similarly can write
\begin{eqnarray*}
\Pi_2(f, g)(x)&=&\int_{\mathbb{R}^{2d}}\sum\limits_{j_1>j_2 + 4} \psi_{j_1}(\xi)\psi_{j_2}(\eta) \mathcal{F}_kf(\xi)\mathcal{F}_kg(\eta)E_k(ix,\xi) E_k(ix,\eta)\,d\mu_k(\xi)d\mu_k(\eta)\\
&=& \int_{\mathbb{R}^{2d}}\sum\limits_{j\in \mathbb{Z}}\psi_j(\xi)\sum\limits_{j_2:j_2< j- 4}\psi_{j_2}(\eta)\mathcal{F}_kf(\xi)\mathcal{F}_kg(\eta)E_k(ix,\xi) E_k(ix,\eta)\,d\mu_k(\xi)d\mu_k(\eta)\\
&=:&\sum\limits_{j\in \mathbb{Z}} \int_{\mathbb{R}^{2d}}\psi^{(2)}_j(\xi)\phi^{(2)}_j(\eta)\mathcal{F}_kf(\xi)\mathcal{F}_kg(\eta)E_k(ix,\xi) E_k(ix,\eta)\,d\mu_k(\xi)d\mu_k(\eta),
\end{eqnarray*}
where $\psi^{(2)}(\xi)=\psi(\xi)$, $\psi^{(2)}_j(\xi)=\psi^{(2)}(\xi/2^{j})$, $\phi^{(2)}(\eta)=\sum\limits_{j<- 4}\psi_j(\eta)$ and $\phi^{(2)}_j(\eta)=\phi^{(2)}(\eta/2^j)$.
\par Then obviously ${\rm supp}\, \phi^{(2)} \subset \{\xi\in\mathbb{R}^d:  |\xi|\leq 2^{-3}\}$ and 
$$\Pi_2(f, g)(x)=\sum\limits_{j\in \mathbb{Z}} \big((\Psi^{(2)}_j*_kf)(x)(\Phi^{(2)}_j*_kg)(x)\big).$$ 
In this case also, we introduce a new function $\theta^{(2)} \in C_c^{\infty}(\mathbb{R}^d )$ such that ${\rm supp}\, \theta^{(2)}\subset\{\xi\in \mathbb{R}^d:2^{-3}\leq |\xi|\leq 2^3\}$, $\theta^{(2)}(\xi)=1$ for $2^{-2}\leq |\xi|\leq 2^2$, and for any $j\in \mathbb{Z}$, define $\theta^{(2)}_j(\xi)=\theta^{(2)} (\xi/2^j)$. We claim that for any $j\in \mathbb{Z}$,
\begin{equation}\label{play with conv}
(\Psi^{(2)}_j*_kf)(x)(\Phi^{(2)}_j*_kg)(x)=\Theta^{(2)}_j *_k\big((\Psi^{(2)}_j*_kf)(\cdot)(\Phi^{(2)}_j*_kg)(\cdot)\big)(x),
\end{equation}
and hence, we can write
\begin{equation*}
\Pi_2(f, g)(x)=\sum\limits_{j\in \mathbb{Z}} \Theta^{(2)}_j *_k\big((\Psi^{(2)}_j*_kf)(\cdot)(\Phi^{(2)}_j*_kg)(\cdot)\big)(x).
\end{equation*}
Similar to the previous case, the proof of \eqref{play with conv} requires showing that
\begin{equation}\label{play with Dunkl conv part2}
(\psi^{(2)}_j(\cdot)\mathcal{F}_kf(\cdot)\big)*_k (\phi^{(2)}_j(\cdot)\mathcal{F}_kg(\cdot)\big)(\xi)= \theta^{(2)}_j(\xi) (\psi^{(2)}_j(\cdot)\mathcal{F}_kf(\cdot)\big)*_k (\phi^{(2)}_j(\cdot)\mathcal{F}_kg(\cdot)\big)(\xi).  
\end{equation}
Here also, since $\psi^{(2)}_j$ is supported in $\{\xi \in \mathbb{R}^d: 2^{j-1}\leq |\xi|\leq 2^{j+1}\}$ and $\phi^{(2)}_j$ is supported in $\{\xi \in \mathbb{R}^d: |\xi|\leq 2^{j-3}\}$, by using Lemma \ref{conv support lemma} one can check that
$${\rm supp}\, (\psi^{(2)}_j(\cdot)\mathcal{F}_kf(\cdot)\big)*_k (\phi^{(2)}_j(\cdot)\mathcal{F}_kg(\cdot)\big) \subset \{\xi \in \mathbb{R}^d: 2^{j-2}\leq |\xi|\leq 2^{j+2}\}.$$
But $\theta^{(2)}_j(\xi)=1$ on $2^{j-2}\leq |\xi|\leq 2^{j+2}$. Hence, the proof of \eqref{play with Dunkl conv part2} is also complete.
\par The proof for $\Pi_3(f, g)$ can be done in exactly the same manner as for $\Pi_2(f, g)$ and is therefore omitted.
 \end{proof}

\subsection{Proofs of the main theorem}\label{sec proof of the main theorem}
In this final section, we first show that the $s$-th fractional power of the Dunkl Laplacian applied to any Schwartz function belongs to $L^p$ for suitable ranges of $s$ and $p$, and then proceed to prove Theorem \ref{frac leib rule main thm}. The next proposition serves as a substitute for the decay estimates associated with the classical fractional Laplacian \cite[Theorem 7.6.2]{GrafakosModernBook} or \cite[Lemma 1]{GrafakosTKPI}.
\begin{prop}\label{thm decay of the fractional Dunkl derivatives}
Given any $f\in \mathcal{S}(\mathbb{R}^d )$ and $s>0$,
we have 
$$|(-\Delta_k)^s(f)(x)|\lesssim |x|^{-(d_k+2s)}$$
for any $|x|>1$.
\end{prop}
\begin{proof}
We recall that the fractional powers of Dunkl Laplacian is given by the formula 
$$(-\Delta_k)^{s}f(x)=\int_{\mathbb{R}^d}|\xi|^{2s}\,\mathcal{F}_kf(\xi)E_k(i\xi,x)\,d\mu_k(\xi).$$
Since the Dunkl transform is a homeomorphism on the space $\mathcal{S}(\mathbb{R}^d)$, and for any Schwartz function $f$, the function $|\cdot|^{2n} \mathcal{F}_k f$ is again a Schwartz function for any $n \in \mathbb{N}$, it suffices to prove the result for $0 < s < 1$.

\par It follows from \eqref{decay esti of heta kerenl with mod} that the heat kernel $h_t(x,y)$ satisfies the Gaussian upper bound, i.e.,  there exists $c>0$ such that 
\[
|h_t(x,y)|\lesssim \frac{1}{\mu_k(B(x,\sqrt t))}\exp\Big(-\frac{d_G(x,y)^2}{ct}\Big)
\]
for $t>0$ and $x,y\in \mathbb R^d$. This, together with \cite[Lemma 2.5]{CD}, implies that the kernel $q_t(x,y)$ of $t\Delta_ke^{ t\Delta_k}$ also satisfies the Gaussian upper bound. That is, there exists $c>0$ such that
\begin{equation}
	\label{eq-Gaussian ub for time derivative of heat kernel}
	|q_t(x,y)|\lesssim \frac{1}{\mu_k(B(x,\sqrt t))}\exp\Big(-\frac{d_G(x,y)^2}{ct}\Big)
\end{equation}
for $t>0$ and $x,y\in \mathbb R^d$.

Fix $x\in \mathbb R^d, |x|>1$. Let $\phi_j \in C^\infty_c(\mathbb R), j=1,2,3;$ be radial functions such that 
$$
{\rm supp}\, \phi_1 \subset \{\xi: |\xi|<1/2\},  \ \ \ {\rm supp}\, \phi_2 \subset \{\xi: 1/4 <|\xi|<4\}, \ \ \ {\rm supp}\, \phi_3 \subset \{\xi: |\xi|>2\}
$$
and
$$
\phi_1(\xi) +\phi_2(\xi)+\phi_3(\xi) = 1, \ \ \xi\in \mathbb R.
$$

 We have
$$
\begin{aligned}
	\left|(-\Delta_k )^sf(x) \right| &=\left| -\frac{1}{\Gamma(1-s)} \int_0^\infty t^{-s} \Delta_ke^{ t\Delta_k}f(x) \, dt \right|\\
	&=\left| -\sum_{j=1}^3\frac{1}{\Gamma(1-s)} \int_0^\infty t^{-s} \Delta_ke^{ t\Delta_k}f_j(x) \, dt \right|\\
	&=: I_1 +I_2+I_3,
\end{aligned}
$$
where $f_j(y)=f(y)\,\phi_j(|y|/|x|)$ for $j=1,2,3$.

For the first term $I_1$, using \eqref{eq-Gaussian ub for time derivative of heat kernel}, ${\rm supp}\, \phi_1 \subset \{\xi: |\xi|<1/2\}$ and the fact $d_G(x,y)\sim |x|$ as $|x|>2|y|$, we obtain
$$
\begin{aligned}
	I_1&\lesssim \int_0^\infty \int_{|x|>2|y|}\frac{1}{t^{s}}\, |q_t(x,y)|\,|f(y)|\,d\mu_k(y)\frac{dt}{t}\\
	&\lesssim \int_0^\infty \int_{|x|>2|y|}\frac{1}{t^{s+d_k/2}} \exp\Big(-\frac{|x|^2}{ct}\Big)|f(y)|\,d\mu_k(y)\frac{dt}{t}\\
	&\lesssim \|f\|_{L^1(d\mu_k)}\int_0^\infty  \frac{1}{t^{s+d_k/2}} \exp\Big(-\frac{|x|^2}{ct}\Big) \frac{dt}{t}\\
	&\lesssim |x|^{-(d_k+2s)}.
\end{aligned}
$$
Similarly,  using \eqref{eq-Gaussian ub for time derivative of heat kernel}, ${\rm supp}\, \phi_3 \subset \{\xi: |\xi|>2\}$ and the fact $d_G(x,y)\sim |y|$ as $|y|>2|x|$, we get
$$
\begin{aligned}
	I_3&\lesssim \int_0^\infty \int_{|y|>2|x|}\frac{1}{t^{s}}\, |q_t(x,y)|\,|f(y)|\,d\mu_k(y)\frac{dt}{t}\\
	&\lesssim \int_0^\infty \int_{|x|>2|y|}\frac{1}{t^{s+d_k/2}} \exp\Big(-\frac{|y|^2}{ct}\Big)|f(y)|\,d\mu_k(y)\frac{dt}{t}\\
	&\lesssim \int_0^\infty \int_{|x|>2|y|}\frac{1}{t^{s+d_k/2}} \exp\Big(-\frac{|x|^2}{ct}\Big)|f(y)|\,d\mu_k(y)\frac{dt}{t}\\
	&\lesssim \|f\|_{L^1(d\mu_k)}\int_0^\infty  \frac{1}{t^{s+d_k/2}} \exp\Big(-\frac{|x|^2}{ct}\Big) \frac{dt}{t}\\
	&\lesssim |x|^{-(d_k+2s)}.
\end{aligned}
$$
For the term $I_2$, we write
$$
\begin{aligned}
	I_2& \leq \frac{1}{\Gamma(1-s)} \int_0^1 t^{-s} |\,e^{ t\Delta_k}[\Delta_kf_2](x)| \, dt + \frac{1}{\Gamma(1-s)} \int_1^\infty t^{-s} |\,\Delta_ke^{ t\Delta_k}f_2(x)| \, dt \\
	&=: I_{21} + I_{22}.
\end{aligned}
$$
Applying \eqref{integration heta kerenl is 1},
$$
\begin{aligned}
	I_{21}&\lesssim  \int_0^1 \|e^{ t\Delta_k}[\Delta_kf_2] \|_{L^\infty}\, t^{-s}  \, dt\\
	&\lesssim \|\Delta_kf_2\|_{L^\infty} \int_0^1 \,t^{-s}  \, dt\\
	&\lesssim \|\Delta_kf_2\|_{L^\infty}.
\end{aligned}
$$
In addition, from the facts $f_2(y) = f(y)\,\phi_2(|y|/|x|)$, ${\rm supp}\, \phi_2 \subset \{\xi: 1/4 <|\xi|<4\}$ and $f\in \mathcal S(\mathbb R^d)$, we have
$$
\|\Delta_kf_2\|_{L^\infty}\lesssim |x|^{-(d_k+2s)}.
$$ 
Similarly, using \eqref{eq-Gaussian ub for time derivative of heat kernel} we also obtain that
$$
\begin{aligned}
	I_{22}&\lesssim  \int_1^\infty \|t\Delta_ke^{ t\Delta_k}f_2\|_{L^\infty}\, \frac{dt}{t^{1+s}}\\
	&\lesssim  \int_1^\infty \|f_2\|_{L^\infty} \,\frac{dt}{t^{1+s}}\\
	&\lesssim \|f_2\|_{L^\infty}\\
	&\lesssim |x|^{-(d_k+2s)}.
\end{aligned}
$$

This completes our proof.
\end{proof}
The following result arises as an immediate corollary of the preceding proposition.
\begin{cor}\label{frac deri in Lp thm}
 Let $f\in \mathcal{S}(\mathbb{R}^d )$. If $s>0$, then $(-\Delta_k)^s(f)\in L^p(\mathbb{R}^d,d\mu_k)$ for any $1\leq p\leq \infty$; and if $2s>\max\{0, d_k/p-d_k\}$, then $(-\Delta_k)^s(f)\in L^p(\mathbb{R}^d,d\mu_k)$ for any $0<p\leq \infty$.
\end{cor}

\par Having established all the necessary ingredients, we are now in a position to present the proofs of our main results. These proofs rely on the boundedness of the paraproducts established in the preceding section.
\begin{proof}[{\bf Proof of Theorem \ref{frac leib rule main thm}}]
Let $\Pi_1$, $\Pi_2$, and $\Pi_3$ be as in Proposition \ref{decompo of fg in paraproduct}. Then, taking Proposition \ref{decompo of fg in paraproduct} into account, the $L^p$-norm of $(-\Delta_k)^s(fg)$ can be estimated by the sum of the $L^p$-norms of paraproduct operators applied to the functions on which the fractional Dunkl Laplacian acts, as follows.
\begin{eqnarray*}
 && \| (-\Delta_k)^s(fg)\|_{L^{p}(d\mu_k)} \\
 &\leq &  \|(-\Delta_k)^s(\Pi_1(f, g))\|_{L^{p}(d\mu_k)} + \|(-\Delta_k)^s(\Pi_2(f, g))\|_{L^{p}(d\mu_k)} + \|(-\Delta_k)^s(\Pi_3(f, g))\|_{L^{p}(d\mu_k)}\\
  &= & \|\widetilde\Pi[\widetilde\theta^{(1)}, \widetilde\psi^{(1)}, \widetilde\phi^{(1)}]((-\Delta_k)^sf,\, g)\|_{L^{p}(d\mu_k)} + \|\Pi[\widetilde\theta^{(2)}, \widetilde\psi^{(2)}, \widetilde\phi^{(2)}]((-\Delta_k)^sf,\, g)\|_{L^{p}(d\mu_k)}\\
  &&+ \|\Pi[\widetilde\theta^{(3)}, \widetilde\psi^{(3)}, \widetilde\phi^{(3)}](f,\,(-\Delta_k)^s g)\|_{L^{p}(d\mu_k)}, 
\end{eqnarray*}
where in the last step we have used the following
\begin{enumerate}[label= (P\arabic*)]
    \item\label{change in paraproducst functions 1} $(-\Delta_k)^s(\Pi_1(f, g))= \widetilde\Pi[\widetilde\theta^{(1)}, \widetilde\psi^{(1)}, \widetilde\phi^{(1)}]((-\Delta_k)^sf,\, g)$, for some $\widetilde\psi^{(1)}, \widetilde\phi^{(1)} \in C_c^\infty(\mathbb{R}^d)$ and a suitable function $ \widetilde\theta^{(1)}$ such that ${\rm supp}\,\widetilde\psi^{(1)} \subset \{\xi \in \mathbb{R}^d: 1/2\leq |\xi|\leq 2\}$, ${\rm supp}\,\widetilde\phi^{(1)} \subset \{\xi \in \mathbb{R}^d: 2^{-5}\leq | \xi|\leq 2^{5}\}$, and ${\rm supp}\,\widetilde\theta^{(1)} \subset \{\xi \in \mathbb{R}^d:  |\xi|\leq 2^{7}\}.$
    \item \label{change in paraproducst functions 2}$(-\Delta_k)^s(\Pi_2(f, g))= \Pi[\widetilde\theta^{(2)}, \widetilde\psi^{(2)}, \widetilde\phi^{(2)}]((-\Delta_k)^sf,\, g)$, for some $\widetilde\psi^{(2)}, \widetilde\phi^{(2)}, \widetilde\theta^{(2)} \in C_c^\infty(\mathbb{R}^d)$ such that ${\rm supp}\,\widetilde\psi^{(2)} \subset \{\xi \in \mathbb{R}^d: 1/2\leq |\xi|\leq 2\}$, ${\rm supp}\,\widetilde\phi^{(2)} \subset \{\xi \in \mathbb{R}^d: |\xi|\leq 2^{-3}\}$, and ${\rm supp}\,\widetilde\theta^{(2)} \subset \{\xi \in \mathbb{R}^d: 2^{-3}\leq |\xi|\leq 2^{3}\}.$
   \item \label{change in paraproducst functions 3}$(-\Delta_k)^s(\Pi_3(f, g))= \Pi[\widetilde\theta^{(3)}, \widetilde\psi^{(3)}, \widetilde\phi^{(3)}](f, \,(-\Delta_k)^sg)$, for some $\widetilde\psi^{(3)}, \widetilde\phi^{(3)}, \widetilde\theta^{(3)} \in C_c^\infty(\mathbb{R}^d)$ such that ${\rm supp}\,\widetilde\psi^{(3)} \subset \{\xi \in \mathbb{R}^d: |\xi|\leq 2^{-3}\}$, ${\rm supp}\,\widetilde\phi^{(3)} \subset  \{\xi\in\mathbb{R}^d: 1/2\leq |\xi|\leq 2\}$, and ${\rm supp}\,\widetilde\theta^{(3)} \subset  \{\xi \in \mathbb{R}^d: 2^{-3}\leq |\xi|\leq 2^{3}\}.$
\end{enumerate}

We present the proof of \ref{change in paraproducst functions 1}; the remaining statements \ref{change in paraproducst functions 2} and \ref{change in paraproducst functions 3}, follow analogously. Indeed, form the definition 
\begin{eqnarray}\label{last step of the proof}
  (-\Delta_k)^s(\Pi_1(f, g))(x)&=&\sum\limits_{j\in \mathbb{Z}} \big((-\Delta_k)^s\Theta^{(1)}_j \big) *_k\big((\Psi^{(1)}_j*_kf)(\cdot)(\Phi^{(1)}_j*_kg)(\cdot)\big)(x)\\
  &=&\sum\limits_{j\in \mathbb{Z}} \widetilde\Theta^{(1)}_j *_k\big(2^{2js}\,(\Psi^{(1)}_j*_kf)(\cdot)(\Phi^{(1)}_j*_kg)(\cdot)\big)(x)\nonumber\\
  &=&\sum\limits_{j\in \mathbb{Z}} \widetilde\Theta^{(1)}_j *_k\big((\widetilde\Psi^{(1)}_j*_k (-\Delta_k)^s(f))(\cdot)(\Phi^{(1)}_j*_kg)(\cdot)\big)(x)\nonumber\\
  &=: & \widetilde \Pi[\widetilde\theta^{(1)}, \widetilde\psi^{(1)}, \widetilde\phi^{(1)}]((-\Delta_k)^sf,\, g)(x),\nonumber
  \end{eqnarray}
 where $\widetilde\theta^{(1)}(\xi) = |\xi|^{2s}\,\theta^{(1)}(\xi)$, $\widetilde\psi^{(1)}(\xi) = |\xi|^{-2s}\,\psi^{(1)}(\xi)$, and $\widetilde\phi^{(1)} = \phi^{(1)}$; and, as in the earlier notation, we define $\widetilde\Psi^{(1)}= \mathcal{F}^{-1}_k \widetilde\psi^{(1)}$, $\widetilde\psi^{(1)}_j(\xi) = \widetilde\psi^{(1)}(\xi/2^j)$, $\widetilde\Psi^{(1)}_j(\xi) = 2^{j d_k} \widetilde\Psi^{(1)}(2^j \xi)$, with analogous notation used for $\widetilde\phi^{(1)}$ and $\widetilde\theta^{(1)}$.
 \par To complete the proof, it remains to prove that each of the operators $\widetilde\Pi[\widetilde\theta^{(1)}, \widetilde\psi^{(1)}, \widetilde\phi^{(1)}]$, $\Pi[\widetilde\theta^{(2)}, \widetilde\psi^{(2)}, \widetilde\phi^{(2)}]$, and $\Pi[\widetilde\theta^{(3)}, \widetilde\psi^{(3)}, \widetilde\phi^{(3)}]$ are bounded operators from $L^{p_1}(\mathbb{R}^d,\,d\mu_k)\times L^{p_2}(\mathbb{R}^d,\,d\mu_k)$ to $L^{p}(\mathbb{R}^d,\, d\mu_k)$.
\par Here, we use the notation $\widetilde\Pi[\widetilde\theta^{(1)}, \widetilde\psi^{(1)}, \widetilde\phi^{(1)}]$ instead of $\Pi[\widetilde\theta^{(1)}, \widetilde\psi^{(1)}, \widetilde\phi^{(1)}]$, since the definition of paraproduct operators is given for functions in $C_c^\infty(\mathbb{R}^d)$, and $\widetilde\theta^{(1)}$ may not be smooth. However, the $L^p$ boundedness of $\widetilde{\Pi}[\widetilde{\theta}^{(1)}, \widetilde{\psi}^{(1)}, \widetilde{\phi}^{(1)}]$ can be established by an argument analogous to that used in the proof of the boundedness of paraproduct operators (Proposition \ref{bddness of general 3paraproducts}), provided that we can show
\begin{equation}\label{max func alt approach}
\sup_{j \in \mathbb{Z}} |\widetilde\Theta^{(1)}_j *_k h(x)| \lesssim M_k (|h|)(x).
\end{equation}
Here, $M_k$ denotes the \emph{Dunkl maximal operator} introduced in \cite{ThangaveluCOMF}, which is known to be bounded on $L^p(\mathbb{R}^d, d\mu_k)$ for all $1 < p < \infty$ (see \cite[Theorem 6.1]{ThangaveluCOMF}).
\par To prove \eqref{max func alt approach}, we observe that by choosing the function $\theta^{(1)}$ to be radial in Proposition \ref{decompo of fg in paraproduct}, the corresponding functions $\Theta^{(1)}$ and $\widetilde\Theta^{(1)}$ also become radial and real valued. We then use the fact that $\widetilde{\Theta} = (-\Delta_k)^s \Theta$ and invoke Proposition \ref{thm decay of the fractional Dunkl derivatives} to obtain the decay
\begin{eqnarray*}
   |\widetilde\Theta^{(1)}(x)|\lesssim (1+ |x|)^{-(d_k+2s)}.
\end{eqnarray*}
Now \eqref{max func alt approach} follows directly from \cite[Theorem 6.2]{ThangaveluCOMF}.
\par Only the boundedness of the operator $\widetilde\Pi[\widetilde\theta^{(1)}, \widetilde\psi^{(1)}, \widetilde\phi^{(1)}]$ posed a challenge, as it requires a separate proof. On the other hand, the boundedness of the other two operators, $\Pi[\widetilde\theta^{(2)}, \widetilde\psi^{(2)}, \widetilde\phi^{(2)}]$ and $\Pi[\widetilde\theta^{(3)}, \widetilde\psi^{(3)}, \widetilde\phi^{(3)}]$, follows directly from Proposition \ref{bddness of general 3paraproducts} and Corollary \ref{frac deri in Lp thm}.
\end{proof}

\section*{Acknowledgments}
 TAB was supported by  the research grant ARC DP260101083 from the Australian Research Council.  SM gratefully acknowledges Sanjay Parui for introducing this problem and B\l{}a\.{z}ej Wr\'{o}bel for clarifying several doubts. SM was supported by Institute Postdoctoral Fellowship from IIT Bombay. The authors thank both reviewers for their careful reading and for their important comments and suggestions which helped to improve the paper.


\end{document}